\documentclass{amsart}

\usepackage[all]{xy}    
\usepackage{graphicx}
\usepackage{amssymb}
\usepackage{mathrsfs}
\usepackage{multirow}
\usepackage{float}
\usepackage{tikz-cd}
\usepackage{adjustbox}
\usepackage{amsthm}
\usepackage{accents}
\usepackage{mathtools}

\usepackage{tikz}
\usetikzlibrary{matrix,arrows}

\newcommand{\ubar}[1]{\underaccent{\bar}{#1}}

\usepackage[numbers,sort&compress]{natbib} 
\usepackage[bookmarksnumbered, bookmarksopen,
colorlinks,citecolor=blue,linkcolor=blue]{hyperref}

\newcounter{RomanNumber}
\newcommand{\MyRoman}[1]{\setcounter{RomanNumber}{#1}\Roman{RomanNumber}}

\newtheorem{theorem}{Theorem}[section]
\newtheorem{lemma}[theorem]{Lemma}
\newtheorem{proposition}[theorem]{Proposition}
\newtheorem{corollary}[theorem]{Corollary}

\theoremstyle{definition}
\newtheorem{definition}[theorem]{Definition}
\newtheorem{example}[theorem]{Example}

\theoremstyle{remark}
\newtheorem{remark}[theorem]{Remark}

\theoremstyle{notation}

\numberwithin{equation}{section}

\begin{document}

\title[Double loop spaces and Cohen groups]{Combinatorics of double loop suspensions, evaluation maps and Cohen groups}

\author{Ruizhi Huang}
\address{Institute of Mathematics, Academy of Mathematics and Systems Science, Chinese Academy of Sciences, Beijing, China, 100190}

\email{huangrz@amss.ac.cn}

\author{Jie Wu}
\address{Department of Mathematics, National University of Singapore,
             Singapore, 119076}

\email{matwuj@nus.edu.sg}

\thanks{
The first named author is supported by Postdoctoral International Exchange Program for Incoming Postdoctoral Students under Chinese Postdoctoral Council and Chinese Postdoctoral Science Foundation. He is also supported in part by Chinese Postdoctoral Science Foundation (Grant No. 2018M631605), and National Natural Science Foundation of China (Grant No. 11801544).
The second named author is partially supported by the Singapore Ministry of Education research grant (AcRF Tier 1 WBS No. R-146-000-222-112) and a grant (No. 11329101) of NSFC of China.}




\numberwithin{theorem}{section}
\begin{abstract}
We reformulate Milgram's model of a double loop suspension in terms of a preoperad of posets, each stage of which is the poset of all ordered partitions of a finite set. Using this model, we give a combinatorial model for the evaluation map and use it to study the Cohen representation for the group of homotopy classes of maps between double loop suspensions. Demonstrating the general theory, we recover Wu's shuffle relations and further provide a type of secondary relations in Cohen groups by using Toda brackets. In particular, we prove certain maps are null-homotopic by combining our relations and the classical James-Hopf invariants.

\end{abstract}

\maketitle
\tableofcontents
\newpage

\section{Introduction}
\noindent In the 1990s, Cohen developed a combinatorial method to study homotopy theory based on the classical James construction. Recall for any based space $(Y,\ast)$, the James construction $J(Y)$ is the free monoid generated by points in $Y$ modulo the relation $\ast=1$. If $Y$ is path connected, then a theorem of James \cite{James} claims that $J(Y)$ is weakly homotopy equivalent to $\Omega \Sigma Y$ and there is a suspension splitting
\[\Sigma \Omega \Sigma Y \simeq \bigvee_{n=1}^{\infty} \Sigma Y^{\wedge n}.\]
Using these nice descriptions,  Cohen \cite{Cohen95} studied the combinatorial structure of $[J(Y), \Omega Z]$ and introduced a universal pro-group $\mathfrak{h}$ (now konwn as a special example of Cohen groups) such that there exists a functorial group homomorphism
\[e_f: \mathfrak{h}\rightarrow [J(Y), \Omega Z]\]
for any given $f: Y\rightarrow \Omega Z$. The importance of the Cohen groups is that $e_f$ is a faithful representation in good cases. For instance, if $Z=\Sigma Y$ with $Y$ a co-$H$-space and $f=E$ is the suspension map, then $\mathfrak{h}$ is isomorphic to the group of natural coalgebra self-transformations of the tensor functor $T$ over $\mathbb{Z}$, $\mathbb{Z}_{(p)}$, or $\mathbb{Q}$ depending on the context \cite{Wu}. Then many classical maps such as Whitehead products, Hopf invariants, power maps and the loop of degree maps can be studied through $\mathfrak{h}$. In particular, this faithful representation can be used to study homotopy exponent problems in unstable homotopy theory (e.g., see \cite{CMW}). Indeed, Cohen's combinatorial program was originally aimed to attack a strong form of Barratt conjecture:
\theoremstyle{plain}
\newtheorem{assump}{Barratt-Cohen Conjecture}
\newenvironment{myassump}[2][]
  {\renewcommand\theassump{#2}\begin{assump}[#1]}
  {\end{assump}}
\begin{myassump}{}\label{Barratt}
Let $f:\Sigma^2 X\rightarrow Z$ be map such that $p^r [f]=0$ in the group $[\Sigma^2 X, Z]$. Then 
\[ \Omega^2 f:\Omega^2\Sigma^2 X\rightarrow \Omega^2 Z \]
has order bounded by $p^{r+1}$ in the group $[\Omega^2\Sigma^2 X,\Omega^2 Z]$.
\end{myassump}
Cohen's program has two steps: first decompose the powers of $[\Omega f]$ as a product of other types of maps (possibly by Cohen groups \cite{Wu, CMW} or the distributivity law \cite{Cohen86}), and second, investigate the group homomorphism 
\[\Omega: [\Omega\Sigma^2 X,\Omega Z]\rightarrow [\Omega^2\Sigma^2 X,\Omega^2 Z]\]
in the hope of showing that some of the factors in the decomposition of powers of $[\Omega f]$ vanish after looping (e.g. \cite{MW}). For this second step, one reasonable approach is to find a suitable normal subgroup $\mathfrak{n}$ of the Cohen group $\mathfrak{h}$ that detects $p^{r+1}[\Omega^2 f]$ and 
fits into the following commutative diagram 
\begin{equation}\label{goal}
\xymatrix{
\mathfrak{h}\ar@{->>}[d] \ar[rr]^{e_f \ \ \ \ \ }
&&
[J(\Sigma X), \Omega Z]
\ar[d]^{\Omega}
\\
\mathfrak{h}/\mathfrak{n}
\ar@{.>}[rr]^{\Omega e_f \ \ \ \ \ }
&&
\lbrack \Omega^2\Sigma^2 X, \Omega^2 Z\rbrack.
}
\end{equation}
In principle, this program is enough for proving or disproving the Barratt-Cohen Conjecture. Indeed, Cohen's program is basically a type of encoding-decoding process, and we only need the functorial information of $[\Omega^2\Sigma^2 X,\Omega^2 Z]$. Furthermore, it is also possible to apply this program for concrete examples, in which case we need to shrink the category of spaces and then add more functorial conditions. 

The diagram (\ref{goal}) is our major goal. In order to obtain such a diagram, it is natural to apply some suitable model of $\Omega^2\Sigma^2 X$ and hope to get nice descriptions of $\lbrack \Omega^2\Sigma^2 X, \Omega^2 Z\rbrack$ and the loop homomorphism $\Omega$. Indeed, general iterated loop suspensions have been widely investigated, and several topologists have constructed models for $\Omega^n\Sigma^n X$. For instance, May \cite{May} constructed an elegant model $\mathcal{C}_n X$ using the little cube operads $\mathcal{C}_n$ and then developed a recognition principle for $n$-fold loop spaces. Smith \cite{Smith} gave a simplicial model $\Gamma^{(n)} X$ by construting a natural filtration of Barratt-Eccles' model $\Gamma X$ \cite{Barratt} of the infinite loop space $\Omega^{\infty}\Sigma^{\infty} X$. In this paper, we will work with Milgram's model \cite{Milgram} which is built up by permutohedra $P_k$ (see Section \ref{Permuto}). 

These various models can be organized into a general framework by using the language of preoperad by Berger \cite{Berger}, which was previously known as a coefficient system in \cite{CMT}. A preoperad is basically an operad without multiplicative structures. In this way, Berger restated Milgram's model as the premonad construction (namely coend) of the preoperad $J^{(2)}$ with 
\[J_k^{(2)}=P_k\times \Sigma_k /\sim,
\]
where $\Sigma_k$ is the $k$-th symmetric group and the equivalence relation identifies certain boundary cells.  However, this form of Milgram's model is not good enough for our combinatorial analysis, for the equivalence relations are defined by using cosets of subgroups of $\Sigma_k$. In order to study unstable homotopy along Cohen's combinatorial program, we then reformulate Milgram's model in terms of particular posets related to $P_k$ \cite{KZ}.
\begin{theorem}[Theorem \ref{config}, Proposition \ref{preoperadJ}, Theorem \ref{Milgram}]
There exists a preoperad of posets $\mathcal{L}$ such that for any connected CW complex $X$,
\[ \Omega^2\Sigma ^2 X\simeq \coprod_k|\mathcal{L}(k)|\times X^{\times k}/\sim,\]
where the right hand side is defined by the usual premonad construction for the geometric realization of $\mathcal{L}$, and each piece $\mathcal{L}(k)$ of $\mathcal{L}$ is the set of all the ordered partitions of a set of size $k$.
\end{theorem}
Through this purely combinatorial description, the group $[\Omega^2\Sigma^2, \Omega^2\Sigma^2]$ of natural transformations are completely determined by the preoperad $\mathcal{L}$. We immediately get the following corollary concerning functorial homotopy decompositions. 
\begin{corollary}
For any idempotent $e: \mathcal{L}\rightarrow \mathcal{L}$ of preoperads of posets, there exists a natural decomposition 
\[\Omega^2\Sigma^2X \simeq E(X) \times I(X),\]
such that 
\[ E(X)\simeq \coprod_k|{\rm Im}(e)_k|\times X^{\times k}/\sim.\]
\end{corollary}
Continuing with Cohen's program, we need to study the loop homomorphism $\Omega: [J(\Sigma X), \Omega Z]\rightarrow \lbrack \Omega^2\Sigma^2 X, \Omega^2 Z\rbrack$, which it is equivalent to studying the evaluation map
\[{\rm ev}: \Sigma\Omega^2\Sigma^2 X\rightarrow \Omega \Sigma^2X,\]
as the adjoint of the identity map. Now let us denote 
\[\mathcal{F}(k)=|\mathcal{L}(k)|, ~~ \ \  \mathcal{F}(X)=\coprod_k\mathcal{F}(k)\times X^{\times k}/\sim,\]
\[ \mathcal{F}_n(X)=\coprod_{1\leq k\leq n}\mathcal{F}(k)\times X^{\times k}/\sim, ~~\ \
D_n(X)=\mathcal{F}_n(X)/\mathcal{F}_{n-1}(X).\]
As we can see from the poset of $\mathcal{F}(k)$,  $\mathcal{F}(k)$ is a $k-1$ dimensional polyhedron with $k!$ top cells, and in general has $k!\cdot {k-1\choose{i}}$ cells of dimension $i$.
Following an elegant explanation of Milgram's model in \cite{Milgram74}, we obtain a clear model of the evaluation map.
\begin{lemma}[Section \ref{evaluation}]\label{mainlemma}
We can choose a filtration preserving evaluation map
\[{\rm ev}: \Sigma\mathcal{F}(X)\rightarrow J(\Sigma X)\]
such that the quotient map 
\[ \bar{{\rm ev}}_n: \Sigma D_n(X)\simeq \Sigma \mathcal{F}(n)^{+}\wedge_{\Sigma_n} X^{\wedge n}\rightarrow J_n(\Sigma X)/J_{n-1}(\Sigma X)\simeq (\Sigma X)^{\wedge n}\]
is the natural projection
\[\Sigma q\wedge {\rm id}: \Sigma \mathcal{F}(n)^{+}\wedge_{\Sigma_n} X^{\wedge n} \rightarrow \Sigma \bigvee_{n!} S^{n-1} \wedge_{\Sigma_n} X^{\wedge n}\simeq  (\Sigma X) ^{\wedge n},\]
where $q$ maps all the cells except the ones of highest dimension to the basepoint.
\end{lemma}
This nice choice of evaluation map allows us to detect certain subgroups $\mathfrak{n}\subseteq \mathfrak{h}$ which act trivially on double loop suspensions (as in \cite{Wu}, we also call such vanishing subgroups relations in Cohen groups). For this purpose, we need a systematic way to handle subgroups of a so-called bi-$\Delta$-group, which is exactly the notion needed in Cohen's combinatorial program. Indeed, bi-$\Delta$-groups were defined and used by the second author \cite{Wu} to detect shuffle relations in Cohen groups. A bi-$\Delta$-group is basically a sequence of groups with both $\Delta$- and co-$\Delta$-structures subject to some natural coherency conditions. For each bi-$\Delta$-group there is a natural notion of a (general) Cohen group (See section \ref{combisection}).
In particular, $[J(Y), \Omega Z]$ and $\mathfrak{h}$ are the Cohen groups of some bi-$\Delta$-groups. Thanks to these notions and examples, Cohen's theory was largely generalized in \cite{Wu}, and is further generalized in this paper. 

Instead of choosing a particular class of maps $[f]$ in $[Y, \Omega Z]$, we start with any group homomorphism 
\[\phi: G\rightarrow [Y^{\wedge (n+1)}, \Omega Z].\]
Indeed, the sequence $\Omega Z^\ast(Y)=[Y^{\times (i+1)}, \Omega Z]_{i\geq 0}$ has a natural bi-$\Delta$-structure such that $[Y^{\wedge (i+1)}, \Omega Z]=\mathcal{Z}_i\Omega Z^\ast(Y)$ is the group of Moore cycles. For any $\Delta$-group $\mathcal{G}$, there is a natural Cohen group $\mathfrak{h}\mathcal{G}$ which is a subgroup of $\mathcal{Z}\mathcal{G}$. It then can be shown that 
\[\mathfrak{h}\Omega Z^\ast(Y)\cong [J(Y), \Omega Z].\] 
In order to study the group $[J(Y), \Omega Z]$ through the representation $\phi$, we then associate a bi-$\Delta$-group $\Phi_nG$ as a bi-$\Delta$-extension of $G$ and extend $\phi$ to a morphism of bi-$\Delta$-groups
\[\phi: \Phi_nG\rightarrow  \Omega Z^\ast(Y),\]
which induces a filtration preserving homomorphism of Cohen groups
\[
\mathfrak{h}\phi: \mathfrak{h}\Phi_n G\rightarrow  [J(Y), \Omega Z],
\]
such that $\mathfrak{h}_n\phi=\phi: G\rightarrow [J_{n+1}(Y), \Omega Z]$. In this context, we then can detect vanishing subgroups $\mathfrak{n}$ by the following procedure.  

\theoremstyle{plain}
\theoremstyle{definition}
\newtheorem*{assum}{Detect relations in Cohen groups}
\newenvironment{myassum}[2][]
  {\renewcommand\theassump{#2}\begin{assum}[#1]}
  {\end{assum}}
\begin{myassum}{}\label{Proce}[Section \ref{Wuapp}, Section \ref{Sh+Co}]
 
\begin{itemize}
 \item Relations for $[J(Y), \Omega Z]$.
  \begin{itemize}
  \item find a sequence of representable elements, subsets, or subgroups $\mathbf{k}$ which are null-homotopic in $\{{\rm Map}(Y^{\times (i+1)}, \Omega Z)\}_{i\geq 0}$.
  \item find representatives of $\mathbf{k}$ and take normal closure to get a sequence of groups $\mathbf{h}$ in $\Phi_nG$. 
  \item take the normal bi-$\Delta$-extension of $\mathbf{h}$ to get a normal bi-$\Delta$-subgroup $\mathcal{H}$ of $\Phi_n G$.
    \item the Cohen group $\mathfrak{h}\mathcal{H}$ consists of all the relations determined by $\mathbf{k}$.
  \end{itemize} 
 \item Relations for $\Omega:  [J(\Sigma X), \Omega Z]\rightarrow \lbrack  \Omega^2\Sigma^2 X, \Omega^2 Z\rbrack$ ($Y=\Sigma X$).
 \begin{itemize}
  \item find a sequence of elements, subsets, or subgroups $\mathbf{s}$ in $ \Omega Z^\ast(\Sigma X)$ which are trivial after looping.
   \item take the pullback of $\mathbf{s}$ along $\phi$ and obtain a sequence of subgroups $\mathbf{r}$ of $\Phi_n G$.
   \item take the normal bi-$\Delta$-extension of $\mathbf{r}$ to get a normal bi-$\Delta$-subgroup $\mathcal{R}$ of $\Phi_n G$.
    \item $\mathfrak{h}\mathcal{R}\cap {\rm Ker}(\Omega\circ \mathfrak{h}\phi)$ consists of all the relations determined by $\mathbf{s}$.
  \end{itemize}
\end{itemize}
\end{myassum}

The procedure for $[J(Y), \Omega Z]$ can be used to cover Cohen's original theory easily; in contrast, the procedure for the loop homomorphism is our main concern here. Thanks to Lemma \ref{mainlemma}, we can detect the shuffle relations of Wu \cite{Wu} in Cohen groups using the cell structure of the permutohedra $P_m$.
\begin{theorem}[Lemma \ref{combish}, Lemma \ref{newnull0}, Proposition \ref{shthm}]
Let 
\[ \beta: \bigvee_{m!} S^{m-2}\rightarrow {\rm sk}_{m-2} \mathcal{F}(m)/ {\rm sk}_{m-3} \mathcal{F}(m)\simeq \bigvee_{(m-1)\cdot m!} S^{m-2}\] 
be the attaching map. 
Then the composite 
\[{\rm sh}\circ\bar{{\rm ev}}_m: \Sigma D_m(X)\twoheadrightarrow (\Sigma X)^{\wedge m}
\longrightarrow \bigvee_{m-1}(\Sigma X)^{\wedge m}\]
is null homotopic. Here, $\Sigma^2\beta\wedge {\rm id}\simeq {\rm sh}$ and the shuffle map ${\rm sh}$ is defined by 
\[{\rm sh}~(y_1\wedge \cdots \wedge y_m)=\sum\limits_{\sigma \in {\rm sh}\subseteq \Sigma_m}y_{\sigma(1)}\wedge \cdots \wedge y_{\sigma(m)}. \]
These null homotopic compositions for $m\geq 2$ determine a group $\mathfrak{b}\mathcal{N}_{\Phi_n G}{\rm Sh}$ such that the diagram
\[
\xymatrix{
\mathfrak{h}\Phi_n G 
\ar[rr]^{\mathfrak{h}\phi}
\ar@{->>}[d]
&&
[J(\Sigma X), \Omega Z]
\ar[d]^{\Omega}
\\
\mathfrak{h}\Phi_n G /\mathfrak{b}\mathcal{N}_{\Phi_n G}{\rm Sh}
\ar@{.>}[rr]^{\Omega\mathfrak{h}\phi}
&&
\lbrack \Omega^2\Sigma^2 X, \Omega^2 Z\rbrack
}
\]
commutes.
\end{theorem} 
We notice that for shuffle relations we only use the cells of $P_m$ in the top two dimensions. The whole cell complex of $P_m$ should provide more relations. Indeed, we can obtain a type of secondary relations using the cells in the top four dimensions in terms of classical Toda brackets.
\begin{theorem}[Section \ref{Todasection}, Theorem \ref{Todarelation}]
The top four dimensions of $P_m$ determine a vanishing Toda bracket, and hence a subgroup 
\[{\rm Indet}_m\subseteq [(\Sigma X)^{\wedge m},\bigvee_{{m-1\choose{3}}}\Sigma^{-1}(\Sigma X)^{\wedge m}]\]
such that 
\[\bar{{\rm ev}}_m^\ast({\rm Indet}_m)=\{0\}\subset [\Sigma D_m(X), \bigvee_{{m-1\choose{3}}}\Sigma^{-1}(\Sigma X)^{\wedge m}].\]
We then have a commutative diagram of groups
\[
\xymatrix{
\mathfrak{h}\Phi_n G 
\ar[rr]^{\mathfrak{h}\phi}
\ar@{->>}[d]
&&
[J(\Sigma X), \Omega Z]
\ar[d]^{\Omega}
\\
\mathfrak{h}\Phi_n G /\mathfrak{b}\mathcal{N}_{\Phi_n G}{\rm T}
\ar@{.>}[rr]^{\Omega\mathfrak{h}\phi}
&&
\lbrack \Omega^2\Sigma^2 X, \Omega^2 Z\rbrack,
}
\]
where $\mathfrak{b}\mathcal{N}_{\Phi_n G}{\rm T}$ is determined by our procedure for ${\rm Indet}_m$.
\end{theorem}

In conclusion, we have constructed the vanishing subgroup 
\[\mathfrak{n}=\langle \mathfrak{b}\mathcal{N}_{\Phi_n G}{\rm Sh},\mathfrak{b}\mathcal{N}_{\Phi_n G}{\rm T}\rangle_{N}\] 
for the key diagram (\ref{goal}). In particular, we have the following proposition which describes a way to construct null-homotopic maps by using the classical James-Hopf invariants.

\begin{proposition}[Proposition \ref{lastnull}]
The loop of any map
\[g : \Omega\Sigma^2X\rightarrow \Omega Z\]
with $[g]\in {\rm Im}~(\mathfrak{h}\phi:\mathfrak{n}\rightarrow [J(\Sigma X), \Omega Z])$ is null homotopic.

Further, for any map $f\in {\rm Sh}_{m}^{N}(\Sigma X; \Omega Z)+{\rm T}_{m}^{N}(\Sigma X; \Omega Z)\subseteq [(\Sigma X)^{\wedge (m+1)}, \Omega Z]$, the loop of the composition
\[\Omega\Sigma^2X \stackrel{H_{m+1}}{\longrightarrow} \Omega\Sigma\big((\Sigma X)^{\wedge(m+1)}\big) \stackrel{J(f)}{\longrightarrow}\Omega Z\]
is null homotopic on the natural $(m+1)$-th filtration $\mathcal{F}_{m+1}(X)$ of $\Omega^2\Sigma^2X\simeq \mathcal{F}(X)$, where $H_{m+1}$ is the James-Hopf invariant, $J(f)$ is the $H$-map such that $J(f)|_{(\Sigma X)^{\wedge(m+1)}}=f$, and ${\rm Sh}^{N}_m(\Sigma X; \Omega Z) (m\geq 1)$ and ${\rm T}_{m}^{N}(\Sigma X; \Omega Z) (m\geq 4)$ are defined to be the normal closures of the subgroups
\[{\rm Sh}_{m}(\Sigma X; \Omega Z)= \lbrack (\Sigma X)^{\wedge (m+1)}\stackrel{{\rm sh}}{\rightarrow}\bigvee_{m}(\Sigma X)^{\wedge (m+1)}\stackrel{g}{\rightarrow}\Omega Z~|~\forall g: \bigvee_{m}(\Sigma X)^{\wedge (m+1)}\rightarrow \Omega Z\rbrack,\]

\[{\rm T}_{m}(\Sigma X; \Omega Z)=
\lbrack (\Sigma X)^{\wedge (m+1)}\stackrel{h}{\rightarrow}\bigvee_{{m\choose{3}}}\Sigma^{-1}(\Sigma X)^{\wedge (m+1)}\stackrel{g}{\rightarrow}\Omega Z~|~\forall~g, {\rm and}~ \forall ~h\in {\rm Indet}_{m+1}\rbrack.\]
in $[(\Sigma X)^{\times (m+1)}, \Omega Z]$ respectively.
\end{proposition}

Computations on homotopy exponent problems may be done by first writing the explicit formula for the shuffle relations and the secondary relations in the Cohen group $\mathfrak{h}\Phi_n G$, and then trying to check if the representative of $p^{r+1}[\Omega^2 f]$ can be expressed as compositions of these relations.

The paper is organized as follows. In Section \ref{Milsection}, we prove our combinatorial reformulation of Milgram's model by showing $\mathcal{L}$ is a preoperad of posets. We also prove some combinatorial aspects of our model. In Section \ref{combisection}, we review Cohen's combinatorial homotopy theory and also its generalization by Wu with our further study. We provide some useful definitions and lemmas to detect relations in Cohen groups which allow us to cover Cohen's original construction easily. In Section \ref{Eva+Shsection}, our aim is to develop a systematic way to detect relations in Cohen groups. We first use our model of a double loop suspension to give a nice description of the evaluation map. Then we turn to study the shuffle map of Wu as an example. We prove that our definition of the shuffle map using the attaching map of a permutohedron coincides with the original definition of Wu up to homotopy. We then construct shuffle relations in Cohen groups to illustrate our method. Section \ref{Todasection} is devoted to higher relations in Cohen groups by using more information of the cell complex of permutohedra. To this end, we introduce the notion of ladder spaces whose associated Toda brackets are always vanishing, and a type of secondary relations for Cohen groups are given. We end the paper with an appendix (Section \ref{AppendA}), where we discuss various aspects of combinatorial James-Hopf operations of abelian bi-$\Delta$-groups and also prove a proposition concerning the loop homomorphism.

\section{Milgram's model}\label{Milsection}
\subsection{Permutohedron and Milgram's combinatorial model of real configuration}\label{Permuto}
Since we adopt a combinatorial approach to study the double loop suspension, we first recall some background information on the permutohedra which serve as the building blocks of Milgram's model. The material in this subsection follows Section $1$ of \cite{KZ}. 

\begin{definition}
\textit{The permutohedron} $P_n$ is the convex hull of the set of points $\{v_\sigma\in\mathbb{R}^n~|~\sigma \in \Sigma_n\}$. 

Here $\Sigma_n$ is the $n$-th symmetric group, and $v_\sigma:=(\sigma^{-1}(1),\sigma^{-1}(2),\cdots, \sigma^{-1}(n))$.
\end{definition}
One can check that $P_n$ is a polytope of dimension $n-1$ which is contained in the hyperplane in $\mathbb{R}^n$ determined by $x_1+x_2+\cdots+x_n=\frac{(n+1)n}{2}$, and its faces of dimension $n-k$ are affinely isomorphic to some $P_{m_1}\times P_{m_2}\times \cdots\times P_{m_k}$.

Now we want to describe the face poset of $P_n$ for which we will use shuffles and unshuffles. 
\begin{definition}
\textit{An unshuffle of a sequence} $\phi=(\phi_1, \phi_2,\cdots,\phi_n)$ of integers is an ordered list of subsequences $s_1$, $s_2$, $\cdots$, $s_k$ of $\phi$ for some $k$ such that their disjoint union is equal to $\phi$. We call $\phi$ \textit{a shuffle of} $s_1$, $s_2$, $\cdots$, $s_k$, and may denote the unshuffle by $s_1|s_2|\cdots|s_k$.
\end{definition}

Let  ${\rm dSh}_{\phi}$ be the set of the unshuffles of $\phi$, we can endow a partial order $\prec$ on ${\rm dSh}_{\phi}$ by removing bars and shuffling the lists. That is, $\prec$ is the transitive closure of the relation
\begin{equation*}
s_1|\cdots|s_{i-1}|s_i|s_{i+1}|s_{i+2}|\cdots|s_k ~\prec~s_1|\cdots|s_{i-1}|h|s_{i+2}|\cdots|s_k, 
\end{equation*}
where $h$ is a shuffle of $s_{i}$ and $s_{i+1}$.

Now the face poset of $P_n$ can be identified with the poset of the unshuffles of $[n]=(1, 2, \cdots, n)$ which we denote by $\mathcal{L}_{\rm id}=({\rm dSh}_{\rm [n]}, \prec)$ (we identify a permutation with the sequence of its images). Under the identification, the vertex $v_\sigma=(\sigma^{-1}(1),\sigma^{-1}(2),\cdots, \sigma^{-1}(n))$ corresponds to $\sigma_1|\sigma_2|\cdots|\sigma_n$. For any face $f$, there is an unique minimal element $a\in \mathcal{L}_{\rm id}$ such that $a \succ b$ for any vertex $b$ of $f$ under the identification. Indeed, $f=P_{m_1}\times P_{m_2}\times \cdots\times P_{m_k}$ corresponds to some unshuffle $s_1|s_2|\cdots|s_k$ with each $s_i$ of length $m_i$. Hence, we may define the dimension (or degree) of $s_1|s_2|\cdots|s_k$ to be $n-k=$ (the length of the unshuffle) $-$ (the number of bars) $-1$. In particular, we may view this $P_n$ as labelled by ${\rm id}\in \Sigma_n$ under the correspondence. Then similarly for any $\sigma_n\in \Sigma_n$, we can define a poset $\mathcal{L}_{\sigma}=({\rm dSh}_{{\rm Im}\sigma}, \prec)$ corresponding to a $P_n$ labelled by $\sigma$. If we view the poset $\mathcal{L}_{\sigma}$ as a category in the standard way, this correspondence gives a geometric realization functor $\mathcal{F}$ from $\mathcal{L}_{\sigma}$ to the category of topological spaces. 

Now we can describe Milgram's model of the real configuration space $F(\mathbb{R}^2, n)$.
\begin{theorem}[Theorem $3.13$ in \cite{BZ}]\label{config}
The geometric realization $\mathcal{F}(n)=\mathcal{F}(\mathcal{L}(n))$ is homeomorphic to a strong deformation retract of $F(\mathbb{R}^2, n)$, where 
\begin{equation*}
\mathcal{L}(n)=\bigcup_{\sigma\in \Sigma_n} \mathcal{L}_\sigma
\end{equation*}
as posets and the geometric realization functor $\mathcal{F}$ is the natural generalization of the one over each $\mathcal{L}_\sigma$. 
\end{theorem}

We notice that by definition, the poset $\mathcal{L}(n)$ consists of all the permutations of length $n$ with any arrangement of possible bars as sets, and there is a natural action of $\Sigma_n$ on $\mathcal{F}(n)$. Furthermore, since the geometric realization is affine, $\mathcal{F}(n)$ inherits a $\Sigma_n$-action which is also free. Indeed, the homotopy in the last theorem can be chosen to be $\Sigma_n$-equivariant by Fox-Neuwirth stratification.

\subsection{Milgram's preoperad and the model of double loop suspensions}
In this subsection, we use the language of preoperads to reproduce Milgram's model for double loop suspensions. One version of this approach is made by Berger in \cite{Berger}, where the permutohedron $P_n$ is viewed as the set of right cosets of subgroups $\Sigma_{i_1}\oplus\Sigma_{i_2}\oplus\cdots\Sigma_{i_k}$ of $\Sigma_n$ with inclusion as the partial order. In contrast, here we use the geometric realization in Section \ref{Permuto} to obtain a combinatorial model.

\begin{definition}	
A \textit{(based) preoperad} with values in a category $\mathcal{C}$ is a contravariant functor $\mathcal{O}: \mathbf{\Lambda}\rightarrow \mathcal{C}$,
where $\mathbf{\Lambda}$ is defined to be the category whose objects are based finite sets $\mathbf{k}=\{0, 1, 2,\ldots, k\}$ with $0$ the based point and whose morphisms are based injective maps. A \textit{map of preoperads} is a natural transformation of functors.

We may write $\mathcal{O}_k$ to be the image of $\mathbf{k}$ and $\phi^\ast: \mathcal{O}_l\rightarrow \mathcal{O}_k$ to be the image of $\phi: \mathbf{k}\rightarrow \mathbf{l}$ under $\mathcal{O}$. 
\end{definition}

\begin{remark}
The terminology preoperad was suggested by Berger according to the fact that an operad is a preoperad by forgetting all the composition operations. This point was observed even much earlier by Cohen, May and Taylor \cite{CMT} who called a preoperad a coefficient system and used it to generalize May's method \cite{May} to a larger context.  
\end{remark}

For the category $\mathbf{\Lambda}$, we notice that for each morphism $\phi\in\mathbf{\Lambda}(\mathbf{k}, \mathbf{l})$, we have a unique decomposition 
\begin{equation*}
  \xymatrix{
  \mathbf{k} \ar[rr]^{\phi} \ar[dr]_{\phi^{\sharp}}& &\mathbf{l} \\
  & \mathbf{k} \ar[ur]_{\phi^{{\rm inc}}} 
  } 
\end{equation*}
such that $\phi^\sharp\in \mathbf{\Lambda}(\mathbf{k}, \mathbf{k})$, $\phi^{{\rm inc}}\in \mathbf{\Lambda}(\mathbf{k}, \mathbf{l})$ and $\phi^\sharp$ is a permutation in $\Sigma_k$ (by forgetting the based point) and $\phi^{{\rm inc}}$ is an increasing map.

\begin{example}[cf. Example $1.5$ in \cite{Berger}]

$(1)$ The collection of configuration spaces defines a topological preoperad $F(\mathbb{R}^n,-): \mathbf{\Lambda}\rightarrow \mathbf{Top}$,  where $F(\mathbb{R}^n,-)$ sends $\mathbf{k}$ to $F(\mathbb{R}^n,k)$, and for any morphism $\phi\in \mathbf{\Lambda}(\mathbf{k}, \mathbf{l})$, we have 
\begin{eqnarray*}
\phi^\ast &:& F(\mathbb{R}^n,l)\rightarrow F(\mathbb{R}^n,k)\\
&&(t_1,t_2,\ldots, t_l) \mapsto (t_{\phi(1)},t_{\phi(2)},\ldots, t_{\phi(k)}).
\end{eqnarray*}
We call $F(\mathbb{R}^n,-)$ defined as above \textit{the configuration preoperad}.

$(2)$ The collection of symmetric groups defines a set-valued preoperad $\Sigma: \mathbf{\Lambda}\rightarrow \mathbf{Sets}$, where $\Sigma$ sends $\mathbf{k}$ to $\Sigma_k$, and for any morphism $\phi\in \mathbf{\Lambda}(\mathbf{k}, \mathbf{l})$, we have 
\begin{eqnarray*}
\phi^\ast &:& \Sigma_l\rightarrow \Sigma_k\\
&&\sigma \mapsto \big((\sigma^{-1}\circ\phi)^\sharp\big)^{-1}.
\end{eqnarray*}
We call $\Sigma$ defined as above \textit{the permutation preoperad}.
\end{example}

In order to obtain a combinatorial model of the double loop suspension, we want to endow a preoperadic structure on the collection of $\mathcal{L}(l)$'s which are $\Sigma_l$-equivariant homotopy deformation retracts of $F(\mathbb{R}^2, l)$'s respectively. As we have pointed out before, $\mathcal{L}(l)$ roughly is $\Sigma_l$ plus bar arrangements. Hence, we introduce the following diagram notation to represent each element $a\in \mathcal{L}(l)$:
\begin{equation*}
  \xymatrix{
 \textbf{l} \ar[r]^{\tilde{a}} \ar@{->>}[d]^{\pi_a} &\textbf{l}\\
 \mathbf{\ubar{t}},
  }
\end{equation*}
where $\tilde{a}$ is obtained by removing all the bars in $a$, $\mathbf{\ubar{t}}=\{1,2,\cdots, t\}$, and the morphism $\pi_a$ is the numbering of ordered subsequences of $a$ (by ignoring the base element $0$). For instance, $235|741|6\in \mathcal{L}(7)$ is presented by $2357416$ and the map $\pi: \mathbf{\ubar{7}}\twoheadrightarrow \mathbf{\ubar{3}}$ defined by $1,2,3\mapsto 1$, $4, 5, 6\mapsto 2$ and $7\mapsto 3$. We see $\pi_a$ is a nondecreasing map in general.

Now for any map $\phi\in \mathbf{\Lambda}(\textbf{k},\textbf{l})$, we have the commutative diagram 
\begin{equation}\label{pullbackJ}
   \xymatrix{
   \textbf{l} \ar[rr]^{\tilde{a}^{-1}} &&\textbf{l} \ar@{->>}[rr]^{\pi_a} &&  \mathbf{\ubar{t}}\\
   \textbf{k} \ar[rr]_{(\tilde{a}^{-1}\circ \phi)^\sharp=\phi^\ast(\tilde{a})^{-1}} \ar[u]_{\phi} &&
   \textbf{k} \ar@{->>}[rr]_{\pi_{\phi^\ast a}} \ar[u]_{(\tilde{a}^{-1}\circ \phi)^{{\rm inc}}}&&
   \mathbf{\ubar{t^\prime}} \ar[u]_{\pi^{{\rm inc}}},
   }
\end{equation}
where in the right square the morphisms $\pi_{\phi^\ast a}$ and $\pi^{{\rm inc}}$ are uniquely determined by the other two maps and $\pi^{{\rm inc}}$ is an increasing map. We then can define a morphism 
\begin{eqnarray*}
\phi^\ast&:&\ \ \  \ \ \ \mathcal{L}(l)\ \  \ \longrightarrow  \ \ \ \ \mathcal{L}(k)\\
  & & a=(\tilde{a}, \pi_a) \longmapsto\phi^\ast a=(\phi^\ast(\tilde{a}), \pi_{\phi^\ast a}).
\end{eqnarray*}

\begin{proposition}\label{preoperadJ}
The above construction defines a preoperad $\mathcal{L}: \mathbf{\Lambda}\rightarrow \mathbf{PSet}$ where $\mathbf{PSet}$ is the category of poset and order preserving morphisms. 
\end{proposition}
\begin{proof}
It is straightforward to check the axioms of a functor once we prove that $\phi^\ast$ is order preserving. Hence we only prove this part. Suppose $a\prec b$ in $\mathcal{L}(l)$, we want to show $\phi^\ast(a)\prec \phi^\ast(b)$ in $\mathcal{L}(k)$. Since we use pairs of morphisms to present elements of $\mathcal{L}(l)$, we need to describe the order relation under this setting. Suppose $a$ and $b$ are respectively presented by
\begin{equation*}
  \xymatrix{
 \textbf{l} \ar[r]^{\tilde{a}} \ar@{->>}[d]^{\pi_a} &\textbf{l} &&  \textbf{l} \ar[r]^{\tilde{b}} \ar@{->>}[d]^{\pi_b} &\textbf{l}&\\
 \mathbf{\ubar{t}}&& {\rm and}&\mathbf{\ubar{s}}&,
  }
\end{equation*}
then it is not hard to show that $a\prec b$ if and only if the following two conditions hold:

$1)$ there exists some nondecreasing surjection $\rho: \mathbf{\ubar{t}}\twoheadrightarrow \mathbf{\ubar{s}}$ such that $\rho\circ \pi_a \circ \tilde{a}^{-1} =\pi_b \circ \tilde{b}^{-1}$;

$2)$ whenever $\pi_a(i)=\pi_a(i+1)$, then $\tilde{b}^{-1}(\tilde{a}(i))<\tilde{b}^{-1}(\tilde{a}(i+1))$.

(Notice that condition $1)$ just means that the subsequences in $a$ can be transformed to that in $b$ by merging some adjacent ones under $\rho$, and condition $2)$ means this transformation preserves the order of elements in each subsequence of $a$.)

Then by hypothesis, we can form the following commutative diagram 

\begin{equation*}
   \xymatrix{
     \textbf{k} \ar[rr]^{\phi} \ar[d]_{(\tilde{a}^{-1}\circ \phi)^\sharp} && 
     \textbf{l} \ar@{=}[rr] \ar[d]^{\tilde{a}^{-1}}  &&
     \textbf{l} \ar[d]^{\tilde{b}^{-1}}  &&
     \textbf{k} \ar[ll]_{\phi}  \ar[d]^{(\tilde{b}^{-1}\circ \phi)^\sharp} \\
     \textbf{k} \ar[rr]_{(\tilde{a}^{-1}\circ \phi)^{{\rm inc}}} \ar@{->>}[d]_{\pi_{\phi^\ast a}} &&
     \textbf{l} \ar@{->>}[d]^{\pi_a}  &&
     \textbf{l} \ar@{->>}[d]^{\pi_b} &&
     \textbf{k} \ar[ll]^{(\tilde{b}^{-1}\circ \phi)^{{\rm inc}}} \ar@{->>}[d]^{\pi_{\phi^\ast b}} \\
     \mathbf{\ubar{t^\prime}} \ar[rr]_{\pi^{{\rm inc}}}  \ar@{.>}@/_1pc/[rrrrrr]_{\rho^\prime}&&
     \mathbf{\ubar{t}} \ar@{->>}[rr]^{\rho} &&
     \mathbf{\ubar{s}}&&
     \mathbf{\ubar{s^\prime}} \ar[ll]^{\pi^{{\rm inc}}},
   }
\end{equation*} 
where the left squares and right squares indicate the effect of $\phi^\ast$ on $a$ and $b$ respectively, and the middle squares reflect condition $1)$.  Our goal is to construct a nondecreasing surjection $\rho^\prime: \mathbf{\ubar{t^\prime}} \twoheadrightarrow \mathbf{\ubar{s^\prime}}$, such that $ \rho \circ \pi^{{\rm inc}}=\pi^{{\rm inc}}\circ \rho^\prime$, which is sufficient to prove the proposition. For if such $\rho^\prime$ exists,  then 
\begin{eqnarray*}
 \pi^{{\rm inc}} \circ \rho^\prime \circ \pi_{\phi^\ast a} \circ \widetilde{\phi^\ast a}^{-1}
&=& \rho \circ \pi^{{\rm inc}}\circ \pi_{\phi^\ast a}  \circ (\tilde{a}^{-1}\circ \phi)^\sharp\\
&=&  \rho \circ \pi_{ a} \circ\tilde{a}^{-1}\circ \phi\\
&=& \pi_{ b} \circ\tilde{b}^{-1}\circ \phi\\
&=&   \pi^{{\rm inc}} \circ \pi_{\phi^\ast b} \circ \widetilde{\phi^\ast b}^{-1}.
\end{eqnarray*}
Since $ \pi^{{\rm inc}}$ is an injection, we have $\rho^\prime \circ \pi_{\phi^\ast a} \circ \widetilde{\phi^\ast a}^{-1}= \pi_{\phi^\ast b} \circ \widetilde{\phi^\ast b}^{-1}$, i.e., condition $1)$ is satisfied. For condition $2)$, we have $\widetilde{\phi^\ast b}^{-1}\circ \widetilde{\phi^\ast a}=
 \big((\tilde{a}^{-1}\circ \phi)^{{\rm inc}}\big)^{-1}\circ \tilde{b}^{-1}\circ \tilde{a}\circ (\tilde{a}^{-1}\circ \phi)^{{\rm inc}}$ where the first and the last maps are increasing maps. Hence condition $2)$ is also satisfied, and we have $\phi^\ast(a)\prec \phi^\ast(b)$.
 
 Now the remainder of the proof is devoted to constructing the required map $\rho^\prime$, for which we introduce two elements of $\mathcal{L}(l)$ related to $a$ and $b$:
\begin{equation*}
  \xymatrix{
 a^\prime:= &\textbf{l} \ar[r]^{\widetilde{a^\prime}=\tilde{b}} \ar@{->>}[d]^{\pi_{a^\prime}=\pi_a} &\textbf{l} &&b^\prime:=&  \textbf{l} \ar[r]^{\widetilde{b^\prime}=\tilde{a}} \ar@{->>}[d]^{\pi_{b^\prime}=\pi_b} &\textbf{l}&\\
 &\mathbf{\ubar{t}}&& {\rm and}&&\mathbf{\ubar{s}}&,
  }
\end{equation*}
We see that $a^\prime$ can be obtained by adding more bars in $b$, and $b^\prime$ can be obtained by removing some bars in $a$. Hence, $a^\prime \prec b$ and $a \prec b^\prime$ (but $a^\prime \not\prec b^\prime$. For instance, let $a=235|71|4|6$ and $b=235|471|6$, then $a^\prime=235|47|1|6$ and $b^\prime=235|714|6$). For $a^\prime$ and $b$, we have the commutative diagram 
\begin{equation*}
   \xymatrix{
   \textbf{l} \ar[rr]^{\tilde{b}^{-1}} &&
   \textbf{l} \ar@{->>}[rr]^{\pi_a} \ar@{->>}@/^1pc/[rrrr]^{\pi_b}&&  
   \mathbf{\ubar{t}}\ar@{->>}[rr]^{\rho} &&
   \mathbf{\ubar{s}}\\
   \textbf{k} \ar[rr]_{(\tilde{b}^{-1}\circ \phi)^\sharp} \ar[u]_{\phi} &&
   \textbf{k} \ar@{->>}[rr]_{\pi_{\phi^\ast a^\prime}} \ar[u]_{(\tilde{b}^{-1}\circ \phi)^{{\rm inc}}} \ar@{->>}@/_1pc/[rrrr]_{\pi_{\phi^\ast b}}&&
   \mathbf{\ubar{t^{\prime\prime}}} \ar[u]_{\pi^{{\rm inc}}} \ar@{.>>}[rr]_{\rho^{\prime\prime}}&&
   \mathbf{\ubar{s^\prime}} \ar[u]_{\pi^{{\rm inc}}},
   }
\end{equation*}
where the unique existence of $\rho^{\prime\prime}$ is ensured by the definition of $\pi$ and the injection $(\tilde{b}^{-1}\circ \phi)^{{\rm inc}}$ (because we may view the effect of $\rho$ as removing bars or merging subsequences, then under $(\tilde{b}^{-1}\circ \phi)^{{\rm inc}}: \textbf{k}\rightarrow \textbf{l}$ the effect of $\rho$ restricts to the effect of some morphism $\rho^{\prime\prime}$ and the commutativity follows from the uniqueness of the morphisms). Similarly, we have the following commutative diagram for $a$ and $b^\prime$:
\begin{equation*}
   \xymatrix{
   \textbf{l} \ar[rr]^{\tilde{a}^{-1}} &&
   \textbf{l} \ar@{->>}[rr]^{\pi_a} \ar@{->>}@/^1pc/[rrrr]^{\pi_b}&&  
   \mathbf{\ubar{t}}\ar@{->>}[rr]^{\rho} &&
   \mathbf{\ubar{s}}\\
   \textbf{k} \ar[rr]_{(\tilde{a}^{-1}\circ \phi)^\sharp} \ar[u]_{\phi} &&
   \textbf{k} \ar@{->>}[rr]_{\pi_{\phi^\ast a}} \ar[u]_{(\tilde{a}^{-1}\circ \phi)^{{\rm inc}}} \ar@{->>}@/_1pc/[rrrr]_{\pi_{\phi^\ast b^\prime}}&&
   \mathbf{\ubar{t^{\prime}}} \ar[u]_{\pi^{{\rm inc}}} \ar@{.>>}[rr]_{\rho^{\prime\prime\prime}}&&
   \mathbf{\ubar{s^{\prime\prime}}} \ar[u]_{\pi^{{\rm inc}}}.
   }
\end{equation*}
Since $a\prec b$, we have $\pi_b\circ\tilde{b}^{-1}\circ\tilde{a}=\pi_b$, which implies $\pi_{\phi^\ast b^\prime}\circ\widetilde{\phi^\ast b}^{-1}\circ\widetilde{\phi^\ast a}=\pi_{\phi^\ast b}$ by combining the above two diagrams together. In particular, we have $\mathbf{\ubar{s^{\prime\prime}}}=\mathbf{\ubar{s^{\prime}}}$ and $\rho^{\prime}=\rho^{\prime\prime\prime}$ is what we want.
\end{proof}

\begin{corollary}[cf. Definition $1.6$ in \cite{Berger}]\label{preoperadF}
There exists a topological preoperad $\mathcal{F}=\mathcal{F}^2: \mathbf{\Lambda}\rightarrow \mathbf{Top_\ast}$ sending $\mathbf{n}$ to $\mathcal{F}(n)$, which is the geometric realization of the preoperad $\mathcal{L}$.
\end{corollary}
\begin{proof}
This corollary follows by the combination of the discussion in Section \ref{Permuto}, Theorem \ref{config} and Proposition \ref{preoperadJ}.
\end{proof}
We can now construct the combinatorial model of double loop suspensions in a standard way which is essentially the well-known Milgram model \cite{Milgram}. 
\begin{theorem}[cf. Proposition $1.7$ in \cite{Berger}]\label{Milgram}
For any connected space $X$,  form the coend $\mathcal{F}(X):=\coprod_{n\geq 1}\mathcal{F}(n)\times X^{\times n}/\sim$ (Definition $2.1$ in \cite{CMT}), where the equivalence relation is specified by 
\begin{equation*}
(\phi^\ast a, (x_1, x_2, \ldots, x_k))\sim (a, \phi_\ast(x_1, x_2, \ldots, x_k)) \ {\rm for} \  \phi\in \mathbf{\Lambda}(\mathbf{k},\mathbf{l}), a\in \mathcal{F}(l),
\end{equation*}
and $\phi_\ast(x_1, x_2, \ldots, x_k)=(x_1^\prime, x_2^\prime, \ldots, x_l^\prime)$ such that $x_{\phi(i)}^\prime=x_i$ and $x_j^\prime=\ast$ if $j\not\in {\rm Im} \phi$.
Then we have 
\begin{equation*}
\mathcal{F}(X)\simeq \Omega^2\Sigma^2X.
\end{equation*}
\end{theorem}
\begin{proof}
As in the proof of Proposition $1.7$ in \cite{Berger}, we want to compare the poset determined by the cell structure of $F(\mathbb{R}^2, l)$ with $\mathcal{L}(l)$.
We start with an anti-lexicographical ordering on $\mathbb{R}^2$. Explicitly, for any two points $x$, $y\in \mathbb{R}^2$, we may define $x=(x_1, x_2)\prec y=(y_1, y_2)$ if and only if either $x_2<y_2$ or $x_2=y_2$ but $x_1<y_1$, and write $x\prec_2 y$ for the first case and $x\prec_1 y$ for the second. Then for any $a\in \mathcal{L}(l)$
\begin{equation*}
  \xymatrix{
 \textbf{l} \ar[r]^{\tilde{a}} \ar@{->>}[d]^{\pi_a} &\textbf{l}\\
\mathbf{\ubar{t}},
  }
\end{equation*}
we can define an associated contractible subspace of $F(\mathbb{R}^2, l)$ by
\begin{equation*}
F^a_l=\{ (t_1, t_2,\ldots t_l)\in F(\mathbb{R}^2, l)~|~t_{\tilde{a}^{}(1)}\mathop{\prec}\limits_{\lambda_{1,2}}t_{\tilde{a}^{}(2)}\mathop{\prec}\limits_{\lambda_{2,3}}\cdots \mathop{\prec}\limits_{\lambda_{l-1,l}}t_{\tilde{a}^{}(l)}  \},
\end{equation*}
where $\lambda_{i,j}\in \{1, 2\}$ and $\lambda_{i,j}=1$ if $\pi_a(i)=\pi_a(j)$, $\lambda_{i,j}=2$ if $\pi_a(i)\neq\pi_a(j)$. 
By the proof of Theorem $3.13$ of \cite{BZ} (Theorem \ref{config}), $\{F^a_l~|~a\in \mathcal{L}(l)\}$ gives an equivariant cellular decomposition for $F(\mathbb{R}^2, l)$, and the reverse of the poset determined by this stratification is precisely $\mathcal{L}(l)$, whose geometric realization (in the sense of \cite{BZ}) is the deformation retract of $F(\mathbb{R}^2, l)$. Hence we get a homotopy equivalence $\mathcal{F}(l)\rightarrow F(\mathbb{R}^2, l)$ for each $l$.

Furthermore, this equivalence is indeed compatible with the preoperad structures. Indeed, by the preoperadic structure of $F(\mathbb{R}^2,-)$, we have 
\begin{equation*}
\phi^\ast(t_1,t_2,\ldots, t_l) =(t_{\phi(1)},t_{\phi(2)},\ldots, t_{\phi(k)}),
\end{equation*}
for any $\phi\in \mathbf{\Lambda}(\mathbf{k},\mathbf{l})$. Then by the diagram (\ref{pullbackJ}) which defines $\phi^\ast a$, we see that $F^{\phi^\ast a}_l$ is the cell exactly corresponding to 
\begin{eqnarray*}
\phi^\ast (F^a_l)=\{(t_{\phi(1)},t_{\phi(2)},\ldots, t_{\phi(k)})\in F(\mathbb{R}^2, k)~|~&t_{\tilde{a}^{}(\phi(1))}&\mathop{\prec}\limits_{\lambda_{\phi(1),\phi(2)}}t_{\tilde{a}^{}(\phi(2))}\mathop{\prec}\limits_{\lambda_{\phi(2),\phi(3)}}\\
&\cdots& \mathop{\prec}\limits_{\lambda_{\phi(k-1),\phi(k)}}t_{\tilde{a}^{}(\phi(k))}\}.
\end{eqnarray*}
Hence, after geometric realization we have a homotopy equivalence of topological preoperads between $\mathcal{F}$ and $F(\mathbb{R}^2, -)$.
Then by Lemma $2.7$ of \cite{CMT}, $\mathcal{F}(X)\simeq F(\mathbb{R}^2, X)$, where the latter is the classical May-Segal model for $\Omega^2\Sigma^2X$ \cite{May}.
\end{proof}

\subsection{Combinatorial structure of $\mathcal{L}(S)$}
In this subsection, we continue to study the preoperad $\mathcal{L}$; the reader may wish to skip this subsection at first reading. It is easy to see that each increasing map $\psi: \mathbf{k}\rightarrow \mathbf{l}$ is generated by elementary so-called \textit{degeneracy operators} $D^i=\sideset{_k}{^i}{\mathop{D}}: \mathbf{k}\rightarrow \mathbf{k+1}$ $(0\leq i\leq k)$ sending $j$ to $j$ for $j\leq i$ and $j$ to $j+1$ for $j\geq i+1$. Hence, any morphism $\phi\in\mathbf{\Lambda}(\mathbf{k}, \mathbf{l})$ can be uniquely written as the composition of a permutation and some degeneracies, and then the category $\mathbf{\Lambda}$ is determined by the collection of symmetric groups $\Sigma_k$'s and degeneracies $_kD^i$ for all $0\leq i\leq k$ subject to some relations. Explicitly, these relations are 
 \begin{equation*}
 D^jD^i= D^{i+1}D^j \ \ \ \ {\rm if} \  \  j\leq i, \ \ 
 \end{equation*}
 \begin{equation*}
 \sigma \circ D^i =D^{\sigma(i+1)-1}\circ d_i\sigma,
 \end{equation*}
 where $\sigma \in \Sigma_{k+1}$ and $d_i\sigma=(\sigma\circ D^i)^\sharp$. Accordingly, we can describe preoperadic structure in terms of the images of these morphisms and the induced relations. For our $\mathcal{L}$, it is easy to check that for any $a\in \mathcal{L}(k+1)$
 \begin{equation*}
  \xymatrix{
   \textbf{k+1} \ar[r]^{\tilde{a}} \ar@{->>}[d]^{\pi_a} &\textbf{k+1}\\
   \mathbf{\ubar{t}},
   }
\end{equation*}
and $\sigma\in \Sigma_{k+1}$, we have $\sigma^\ast(a)=(\sigma^{-1}\circ \tilde{a}, \pi_a)$ (hence we may also write $\sigma^\ast(a)=\sigma^{-1}\circ a$), and the commutative diagram 
\begin{equation}\label{defDi} 
   \xymatrix{
   \textbf{k+1} \ar[rr]^{\tilde{a}^{-1}} &&
   \textbf{k+1} \ar@{->>}[rr]^{\pi_a} && 
   \mathbf{\ubar{t}}\\
   \textbf{k} \ar[rr]_{d_i(\tilde{a}^{-1})} \ar[u]_{D^i} &&
   \textbf{k} \ar@{->>}[rr]_{\pi_{a}^i} \ar[u]_{D^{\tilde{a}^{-1}(i+1)-1}}&&
   \mathbf{\ubar{t^\prime}} \ar[u]_{\pi^{{\rm inc}}}
   }
\end{equation}
 shows that $D^{i\ast}(a)=(d_i(\tilde{a}^{-1})^{-1}, \pi_a^i)$.
 
 Now we turn to study the combinatorial coend $\mathcal{L}(S)=\coprod_{n\geq 1}\mathcal{L}(n)\times S^{\times n}/\sim$ for any based set $S$ where the equivalence relation is similarly defined as that of $\mathcal{F}(X)$ in Theorem \ref{Milgram}. We start with constructing some special morphisms $e_{i, j}^{\epsilon}: \mathcal{L}(k)\rightarrow \mathcal{L}(k+1)$, $0\leq i, j\leq k$ and $\epsilon\in \{-1, 0, 1\}$ by the following commutative diagram:
 \begin{equation*}
   \xymatrix{
   \textbf{k+1} \ar[rr]^{e_{i, j}^{}(\tilde{b}^{-1})} &&
   \textbf{k+1} \ar@{->>}[rr]^{\pi_{e_{,j}^{\epsilon}(b)}} && 
   \mathbf{\underline{s^\prime}}\\
   \textbf{k} \ar[rr]_{\tilde{b}^{-1}} \ar[u]_{D^i} &&
   \textbf{k} \ar@{->>}[rr]_{\pi_b} \ar[u]_{D^j}&&
   \mathbf{\ubar{s}} \ar[u]_{\pi^{{\rm inc}}},
   }
\end{equation*}
i.e., $e_{i, j}^{\epsilon}(b)=((e_{i, j}^{}(\tilde{b}^{-1}))^{-1}, \pi_{e_{,j}^{\epsilon}(b)})$. 
Explicitly, $e_{i, j}^{}(\tilde{b}^{-1})$ maps $i+1$ to $j+1$, and there are several cases for defining $\pi_{e_{,j}^{\epsilon}(b)}$.
\begin{enumerate}
   \item $0<j<k$
       \begin{enumerate}
           \item  $\epsilon=0$
                   \begin{enumerate} 
                       \item If $\pi_b(j)=\pi_b(j+1)$, then $\pi_{e_{,j}^{\epsilon}(b)}(j)=\pi_{e_{,j}^{\epsilon}(b)}(j+1)=\pi_{e_{,j}^{\epsilon}(b)}(j+2)$, $s^\prime=s$ and $\pi^{{\rm inc}}$ is identity.
                       \item If $\pi_b(j)=\pi_b(j+1)-1$, then $\pi_{e_{,j}^{\epsilon}(b)}(j)+1=\pi_{e_{,j}^{\epsilon}(b)}(j+1)=\pi_{e_{,j}^{\epsilon}(b)}(j+2)-1$, $s^\prime=s+1$.
                  \end{enumerate}               
         \item $\epsilon=1$
                  \begin{enumerate}              
                      \item If $\pi_b(j)=\pi_b(j+1)$, defined as in $(1.a.i)$.
                      \item If $\pi_b(j)=\pi_b(j+1)-1$, then $\pi_{e_{,j}^{\epsilon}(b)}(j)=\pi_{e_{,j}^{\epsilon}(b)}(j+1)=\pi_{e_{,j}^{\epsilon}(b)}(j+2)-1$, $s^\prime=s$.
                 \end{enumerate}
         \item  $\epsilon=-1$
              \begin{enumerate} 
                  \item If $\pi_b(j)=\pi_b(j+1)$, defined as in $(1.a.i)$.
               \item If $\pi_b(j)=\pi_b(j+1)-1$, then $\pi_{e_{,j}^{\epsilon}(b)}(j+1)=\pi_{e_{,j}^{\epsilon}(b)}(j+2)=\pi_{e_{,j}^{\epsilon}(b)}(j)+1$, $s^\prime=s$.
            \end{enumerate}
     \end{enumerate}
  \item $j=k$, there are only two possible values for $\epsilon$.
       \begin{enumerate}
            \item $\epsilon=0$, then $\pi_{e_{,j}^{\epsilon}(b)}(k)+1=\pi_{e_{,j}^{\epsilon}(b)}(k+1)$, $s^\prime=s+1$.
             \item $\epsilon=1$, then $\pi_{e_{,j}^{\epsilon}(b)}(k)=\pi_{e_{,j}^{\epsilon}(b)}(k+1)$, $s^\prime=s$.
      \end{enumerate}
  \item $j=0$, there are only two possible values for $\epsilon$.
       \begin{enumerate}
            \item $\epsilon=0$, then $\pi_{e_{,j}^{\epsilon}(b)}(1)+1=\pi_{e_{,j}^{\epsilon}(b)}(2)=2$, $s^\prime=s+1$.
            \item $\epsilon=-1$, then $\pi_{e_{,j}^{\epsilon}(b)}(1)=\pi_{e_{,j}^{\epsilon}(b)}(2)=1$, $s^\prime=s$.
       \end{enumerate}
\end{enumerate}
These morphisms can be interpreted by Figure \ref{combinmap} where the numbers refer to the locations. The numbers $j$ and $j+1$ lie in the same box if and only if they have same image under the morphism $\pi_b$. The left image presents the case $(1.a.i)$ while the right one presents the case $(1.a.ii)$. We should notice that the morphisms $e_{i,j}^{0}$, $e_{i,k}^{1}$ and $e_{i,0}^{-1}$ preserve the partial order while the others do not. For the simplicity of notation, we may further define $e_{i,k}^{-1}=e_{i,k}^{0}$ and $e_{i,0}^{1}=e_{i,0}^{0}$. Also, $D^{i\ast}\circ e_{i,j}^{\epsilon}={\rm id}$ and therefore $e_{i,j}^{\epsilon}$ is injective. 
\begin{figure}[H]
\centering
\includegraphics[width=4.2in]{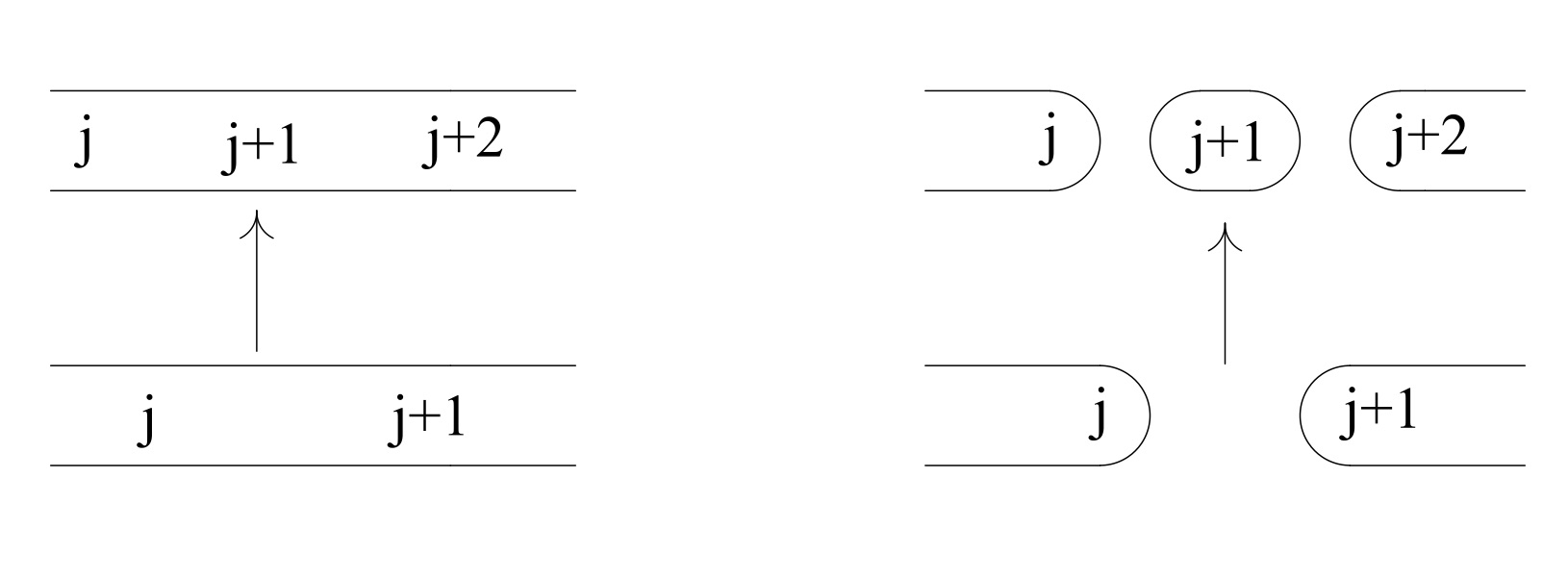}
\caption{The morphism $e_{i, j}^{0}$ with $0<j<k$}\label{combinmap}
\end{figure}
On the other hand, we also have the morphism $D_\ast^i: S^{\times k}\rightarrow S^{\times (k+1)}$ defined by 
\begin{equation*}
D_\ast^i(x_1,x_2,\ldots,x_k)=(x_1,x_2,\ldots, x_i,\ast, x_{i+1},\ldots,x_k).
\end{equation*}
Then there exist morphisms $ e_{i, j}^{\epsilon}\times D_\ast^i: \mathcal{L}(k)\times S^{\times k} \rightarrow \mathcal{L}(k+1)\times S^{\times (k+1)}$, and we want to show that each of these morphisms induces a morphism between the sets of the orbits $\mathcal{L}(k)\times_{\Sigma_k}S^{\times k}$ and $\mathcal{L}(k+1)\times_{\Sigma_{k+1}}S^{\times (k+1)}$.

\begin{lemma}\label{orbits}
There exists a map $f_{i, j, \epsilon}$ such that the diagram 
\begin{equation*}
    \xymatrix{
     \mathcal{L}(k)\times S^{\times k} \ar[rr]^{e_{i, j}^{\epsilon}\times D_\ast^i} \ar@{->>}[d]^{\pi}&&
      \mathcal{L}(k+1)\times S^{\times (k+1)}\ar@{->>}[d]^\pi\\
      \mathcal{L}(k)\times_{\Sigma_k}S^{\times k} \ar@{.>}[rr]^{f_{i, j, \epsilon}} &&
      \mathcal{L}(k+1)\times_{\Sigma_{k+1}}S^{\times (k+1)}
   }
\end{equation*}
commutes.
\end{lemma}
\begin{proof}
By construction, we may view $\Sigma_k$ acting diagonally on $\mathcal{L}(k)\times S^{\times k}$ by
\begin{equation}\label{krelation}
\sigma\cdot (b,(x_1,x_2,\ldots, x_k))=(\sigma^{-1}\circ b, (x_{\sigma(1)}, x_{\sigma(2)},\ldots, x_{\sigma(k)})),
\end{equation}
for any $\sigma\in \Sigma_k$, $b\in \mathcal{L}(k)$ and $x_i\in S$.
Then
\begin{eqnarray*}
&(e_{i, j}^{\epsilon}\times D_\ast^i )&(\sigma^{-1}\circ b, (x_{\sigma(1)}, x_{\sigma(2)},\ldots, x_{\sigma(k)}))=\\
 &&\ \ \  (e_{i, j}^{\epsilon}(\sigma^{-1}\circ b),(x_{\sigma(1)}, x_{\sigma(2)},\ldots,x_{\sigma(i)},\ast,x_{\sigma(i+1)}, \ldots, x_{\sigma_k})).
\end{eqnarray*}
We then define a permutation $\tilde{\sigma}^i\in \Sigma_{k+1}$ such that the diagram 
\begin{equation}\label{sigmai}
   \xymatrix{
   \mathbf{k} \ar[rr]^{D^i} \ar[d]^{\sigma} &&
   \mathbf{k+1}\ar[d]^{\tilde{\sigma}^i} \\
   \mathbf{k} \ar[rr]^{D^i} &&
   \mathbf{k+1}
   }
\end{equation}
commutes. Hence we have 
\begin{eqnarray*}
\tilde{\sigma}^i D^i_\ast(x_1,x_2\ldots, x_k)
&=& \tilde{\sigma}^i (x_1,x_2,\ldots, x_i,\ast, x_{i+1},\ldots,x_k)\\
&=& (x_{\sigma(1)}, x_{\sigma(2)},\ldots,x_{\sigma(i)},\ast,x_{\sigma(i+1)}, \ldots, x_{\sigma_k}),
\end{eqnarray*}
and the commutative diagram 
\begin{equation*}
    \xymatrix{ 
    \mathbf{k+1}\ar[rr]_{\tilde{\sigma}^i } \ar@/^1pc/[rrrr]^{e_{i, j}^{}(\tilde{b}^{-1}\circ \sigma)} &&
    \mathbf{k+1}\ar[rr]_{e_{i, j}^{}(\tilde{b}^{-1})} &&
    \mathbf{k+1}\ar@{->>}[rr]^{\pi_{{e_{,j}^{\epsilon}}}(b)=\pi_{{e_{,j}^{\epsilon}}}(\sigma^{-1}\circ b)}&&
    \mathbf{\underline{s^\prime}} \\
    \mathbf{k}\ar[rr]_{\sigma} \ar[u]^{D^i}&&
    \mathbf{k}\ar[rr]_{\tilde{b}^{-1}}  \ar[u]^{D^i}&&
    \mathbf{k}\ar@{->>}[rr]_{\pi_b=\pi_{\sigma^{-1}\circ b}}  \ar[u]^{D^j}&&
    \mathbf{\underline{s}} \ar[u]^{\pi^{{\rm inc}}},
    }
\end{equation*}
which implies 
\begin{eqnarray*}
e_{i, j}^{\epsilon}(\sigma^{-1}\circ b)
&=& (( e_{i, j}^{}(\tilde{b}^{-1}\circ \sigma))^{-1}, \pi_{e_{, j}^{\epsilon}}(b))\\
&=&((\tilde{\sigma}^i)^{-1}\circ (e_{i, j}^{}(\tilde{b}^{-1}))^{-1},\pi_{e_{, j}^{\epsilon}}(b))\\
&=&(\tilde{\sigma}^i)^{-1}\circ e_{i, j}^{\epsilon}(b).
\end{eqnarray*}
By combining the above together, we have 
\begin{eqnarray*}
(e_{i, j}^{\epsilon}\times D_\ast^i )(\sigma\cdot (b,(x_1,x_2,\ldots, x_k)))
 &=&(e_{i, j}^{\epsilon}\times D_\ast^i )(\sigma^{-1}\circ b, (x_{\sigma(1)}, x_{\sigma(2)},\ldots, x_{\sigma(k)}))\\
 &=& ((\tilde{\sigma}^i)^{-1}\circ e_{i, j}^{\epsilon}(b), \tilde{\sigma}^i D^i_\ast(x_1,x_2\ldots, x_k) )\\
 &=& \tilde{\sigma}^i\cdot \big((e_{i, j}^{\epsilon}\times D_\ast^i )(b, (x_1,x_2\ldots, x_k) )\big),
 \end{eqnarray*}
 which shows that $f_{i, j,  \epsilon}$ can be well-defined.
\end{proof}
We now want to study the structure of $\mathcal{L}(S)$ by using the morphisms $f_{i,j,\epsilon}$ of Lemma \ref{orbits}.
\begin{proposition}\label{cohenJ}
Let $\mathcal{L}_{k+1}(S)=\coprod_{1\leq n\leq k+1}\mathcal{L}(n)\times S^{\times n}/\sim$ be the $(k+1)$-st `skeleton' of $\mathcal{L}(S)$, then the coequalizer of the well-defined morphisms $f_{i, j ,\epsilon}$ for all $0\leq i \leq k$ and all possible $\epsilon \in \{0,\pm 1\}$ is isomorphic to $\mathcal{L}_{k+1}(S)$. Moreover, the sequence 
\[
\begin{tikzcd}
0 \ar[r]
&
\mathcal{L}(k)\times_{\Sigma_k}S^{\times k} \ar[r, shift left=1.52ex,"f_{0,0,\epsilon}"]
    \arrow{r}[description]{\vdots}
    \ar[r, shift right=1.52ex, swap, "f_{k, k, \epsilon}"]
&
\mathcal{L}(k+1)\times_{\Sigma_{k+1}}S^{\times (k+1)} \ar[r,"q"] 
&
\mathcal{L}_{k+1}(S) \ar[r]
&
0
\end{tikzcd}
\]
is exact.
\end{proposition} 
 \begin{proof}
 First, we notice that $q$ is surjective by Lemma $2.2$ of \cite{CMT}. Also by definition, the relation for defining $\mathcal{L}_{k+1}(S)$ is 
 \begin{equation*}
(\phi^\ast a, (x_1, x_2, \ldots, x_k))\sim (a, \phi_\ast(x_1, x_2, \ldots, x_k)) \ {\rm for} \  \phi\in \mathbf{\Lambda}(\mathbf{k},\mathbf{l}), a\in \mathcal{L}(l).	
\end{equation*}
Specializing to the case $l=k+1$, we see the relation can be written as
\begin{equation}\label{Di}
(D^{i\ast} a, (x_1, x_2, \ldots, x_k))\sim (a, (x_1, x_2, \ldots, x_i,\ast, x_{i+1},\ldots, x_k)),
\end{equation}
for each $0\leq i\leq k$ and $a\in \mathcal{L}(k+1)$. For this $a$, there exists a unique $j$ such that $\tilde{a}(j+1)=i+1$. Hence, we see that $e_{i,j}^{\epsilon}(D^{i\ast}a)=a$ for some $\epsilon$. Then by Lemma \ref{orbits} we have a commutative diagram
\[
\begin{tikzcd}
0 \ar[r]
&
\mathcal{L}(k)\times S^{\times k} 
     \ar[dd, two heads]
    \ar[r, shift left=1.52ex,"e_{0, 0}^{\epsilon}\times D_\ast^0"]
    \arrow{r}[description]{\vdots}
    \ar[r, shift right=1.52ex, swap, "e_{k, k}^{\epsilon}\times D_\ast^k"]
&
\mathcal{L}(k+1)\times S^{\times (k+1)} \ar[r,"\tilde{q}"] 
 \ar[dd, two heads]
&
\widetilde{\mathcal{L}}_{k+1}(S) \ar[r]
 \ar[dd, two heads]
&
0 \\ &&& &
\\ 
0 \ar[r]
&
\mathcal{L}(k)\times_{\Sigma_k}S^{\times k} \ar[r, shift left=1.52ex,"f_{0, 0, \epsilon}"]
    \arrow{r}[description]{\vdots}
    \ar[r, shift right=1.52ex, swap, "f_{k,k,\epsilon}"]
&
\mathcal{L}(k+1)\times_{\Sigma_{k+1}}S^{\times (k+1)} \ar[r,"q"] 
&
\mathcal{L}_{k+1}(S) \ar[r]
&
0,
\end{tikzcd}
\]
where the first row is defined to be exact. The injectivity of $f_{i,j,\epsilon}$ can be argued as follows. Suppose we have $f_{i,j,\epsilon}[a, x]=f_{i,j,\epsilon}[b, y]$ for some $a$, $b\in \mathcal{L}(k)$ and $x$, $y\in S^{\times k}$. Then by the diagram there exists $\tau \in \Sigma_{k+1}$ such that $\tau \cdot (e_{i,j}^{\epsilon}(a), D_\ast^i(x))=(e_{i,j}^{\epsilon}(b), D_\ast^i(y))$. In particular, we have 
\[\tau (x_1,\ldots, x_i,\ast,x_{i+1},\ldots, x_k)=(y_1,\ldots, y_i,\ast, y_{i+1},\ldots, y_k),\]
which implies $\tau=\tilde{\sigma}^i$, $\sigma (x)=y$ for some $\sigma\in \Sigma_{k}$ in the sense of diagram \ref{sigmai}. Then $\tau\cdot e_{i,j}^{\epsilon}(a)=\tilde{\sigma}^i \cdot e_{i,j}^{\epsilon}(a)=e_{i,j}^{\epsilon}(\sigma^\ast (a))$ by the proof of Lemma \ref{orbits}. Then $e_{i,j}^{\epsilon}(\sigma^\ast (a))=e_{i,j}^{\epsilon}(b)$ implies $\sigma^\ast (a)=b$ and $[a, x]=[\sigma a,\sigma x]=[b, y]$. Hence, $f_{i,j,\epsilon}$ is injective. Similar arguments will show that the second row of the diagram is exact. 
 \end{proof}
 \begin{lemma}
 The collection of morphisms $\big\{f_{i,j,\epsilon} ~|~ 0\leq i, j \leq k, \epsilon \in \{\pm1, 0\}\big\}$ satisfies the following relations:
\begin{eqnarray*}
f_{i,j,\epsilon} \circ f_{i^\prime,j^\prime,\epsilon^\prime}= f_{i^\prime+1,j^\prime+1,\epsilon^\prime} f_{i,j,\epsilon}  \ \ \  {\rm if}~ i\leq i^\prime , j \leq j^\prime; \\
f_{i,j+1,\epsilon} \circ f_{i^\prime,j^\prime,\epsilon^\prime}= f_{i^\prime+1,j^\prime,\epsilon^\prime} f_{i,j,\epsilon}   \ \ \  {\rm if}~ i\leq i^\prime , j \geq j^\prime.
\end{eqnarray*}
 \end{lemma}
 \begin{proof}
 These relations follow from the corresponding relations for $e_{i,j}^{\epsilon}$ and $D^i_\ast$ which can be easily obtained from their definitions and the relation
 \begin{equation*}
 D^jD^i= D^{i+1}D^j \ \ \ \ {\rm if} \  \  j\leq i.\ \ 
 \end{equation*}
 \end{proof}
 
 \section{Combinatorial homotopy theory: Cohen groups}\label{combisection}
\noindent In this section, we review Cohen's combinatorial homotopy theory and Wu's bi-$\Delta$-group approach with our further generalization. The material here presents a basic way to study homotopy exponent problems.
 \subsection{Cohen's original approach}\label{Cohensection}(\cite{Cohen95}) Given $f: Y \rightarrow \Omega Z$, we can construct some canonical maps in $[Y^{\times (n+1)}, \Omega Z]$. Define 
 \begin{equation*}
 Y^{\times (n+1)}\stackrel{\pi_i}{\longrightarrow} Y \stackrel{f}{\longrightarrow} \Omega Z,
 \end{equation*}
where $\pi_i(y_1, y_2,\ldots, y_{n+1})=y_i$. We may denote the homotopy class of this map by $y_i$. Then there is a natural representation 
\begin{equation*}
e_f: F_{n+1}=F_{n+1}(x_1, x_2, \ldots, x_{n+1}) \rightarrow [Y^{\times (n+1)}, \Omega Z]
\end{equation*}
defined by $e_f(x_i)=y_i$, where $F_{n+1}(x_1, x_2,\ldots, x_{n+1})$ is the free group of rank $n+1$ generated by $x_1, x_2, \ldots, x_{n+1}$. This group homomorphism may have nontrivial kernel depending on the choice of the maps and spaces involved, and there are two typical cases: 
\begin{itemize}
\item If the reduced diagonal $\bar{\Delta}: Y\rightarrow Y\wedge Y$ is null homotopic (for instance when $Y$ is a co-$H$-space), then the iterated commutator $[[y_{i_1}, y_{i_2}], \ldots, y_{i_t}]=1$ in $[Y^{\times (n+1)}, \Omega Z]$ when $y_{i_a}=y_{i_b}$ for some $a\neq b$. Then 
\[N_{n+1}=\langle[[x_{i_1}, x_{i_2}], \ldots, x_{i_t}] ~|~y_{i_a}=y_{i_b} ~{\rm for }~{\rm some}~a\neq b\rangle_N\unlhd F_{n+1}\]
 lies in the kernel of $e_f$.
\item If $p^r[f]=0$, then $y_i^{p^r}=1$ for each $i$. Then 
\[ N_{n+1}=\langle x_1^{p^r},\ldots, x_{n+1}^{p^r}\rangle_N \unlhd {\rm Ker}(e_f). \]
\end{itemize}
On the other hand, the sequence of sets $\{[Y^{\times (n+1)}, \Omega Z]\}_{n\geq 0}$ can be endowed with a $\Delta$-group structure induced from the co-$\Delta$ structure of the sequence of spaces $\{Y^{\times (n+1)}\}_{n\geq 0}$, the structural morphisms of which are defined by for any $0\leq i\leq n$
\begin{equation*}
d^i: Y^{\times n}\rightarrow  Y^{\times (n+1)}, ~ \ \ \  \ (y_1, y_2, \ldots, y_n) \mapsto (y_1, \ldots, y_i, \ast, y_{i+1}, \ldots, y_n).
\end{equation*}
Similarly, there is a $\Delta$-group structure on $\{F_{n+1}\}_{n\geq 0}$ defined by  for any $0\leq i\leq n$ 
\begin{equation*}
\tilde{d}_i: F_{n+1}\rightarrow F_n, \ \ \ \ \ 
\tilde{d}_i(x_j)=\left\{\begin{array}{ll}
x_j& j\leq i,\\
1& j=i+1,\\
x_{j-1}&j\geq i+2.
\end{array}
\right.
\end{equation*}
By direct computations, there is a commutative diagram
\begin{equation*}
 \xymatrix{
 N_{n+1} \ar@{^{(}->}[rr] \ar[d]^{\tilde{d}_{i\vert}} &&
 F_{n+1}  \ar@{->>}[rr] \ar[d]^{\tilde{d}_i} \ar@/^1pc/[rrrr]^{e_f} &&
 K_{n+1} \ar@{.>}[rr]_{e_f \ } \ar[d]^{\bar{d}_i} &&
 [Y^{\times (n+1)}, \Omega Z]  \ar[d]^{d^{i\ast}}\\
 N_{n} \ar@{^{(}->}[rr]  &&
 F_{n}  \ar@{->>}[rr]  \ar@/_1pc/[rrrr]_{e_f} &&
 K_{n} \ar@{.>}[rr]^{e_f \ }  &&
 [Y^{\times n}, \Omega Z] , 
 }
\end{equation*}
where $N_n$ is chosen depending on the condition, $K_n=F_n/N_n$. Now the sequence of groups $\{K_{n+1},\bar{d}_i\}_{n\geq 0}$ is a $\Delta$-group, and we have commutative diagram 
\[
\begin{tikzcd}
0 
  \ar[r]
&
\lbrack J_{n+1}(Y), \Omega Z \rbrack
  \ar[r, "q_{n+1}^{\ast}"]
&
\lbrack Y^{\times (n+1)}, \Omega Z \rbrack
    \ar[r, shift left=1.52ex,"d^{0\ast}"]
    \arrow{r}[description]{\vdots}
    \ar[r, shift right=1.52ex, swap, "d^{n\ast}"]
&
\lbrack Y^{\times n}, \Omega Z \rbrack
  \ar[r] 
&
0 \\ 
&&& &
\\ 
0 
\ar[r]
&
\mathfrak{h}_{n+1} 
 \ar[r]
 \ar[uu, dashrightarrow, "e_f"]
&
K_{n+1} \ar[r, shift left=1.52ex,"\bar{d}_0"]
    \arrow{r}[description]{\vdots}
    \ar[r, shift right=1.52ex, swap, "\bar{d}_n"]
    \ar[uu,"e_f"]
&
K_n 
\ar[r] 
\ar[uu,"e_f"]
&
0,
\end{tikzcd}
\]
where $q_{n+1}$ is the natural projection, $q_{n+1}^\ast$ is an injection by Corollary $1.1.4$ of \cite{Wu}, $[ J_{n+1}(Y), \Omega Z]$ is the equalizer of $\{ d^{i\ast}~|~0\leq i\leq n\}$ by Lemma $2.9$ of \cite{Wu2003}, and $\mathfrak{h}_{n+1}$ is defined to be  the equalizer of $\{ \bar{d}_{i}~|~0\leq i\leq n\}$. Then there is an induced morphism $\bar{p}_{n+1}= \bar{d}_{i\vert}: \mathfrak{h}_{n+1}\rightarrow \mathfrak{h}_{n}$ which is indeed an epimorphism for each $n$. Similarly, we have an epimorphism $p_{n+1}: [ J_{n+1}(Y), \Omega Z] \twoheadrightarrow [ J_{n}(Y), \Omega Z]$ for each $n$. Now by the above diagram we have the so-called \textit{Cohen representation} at infinity
\begin{equation*}
e_f: \mathfrak{h}=\lim_{n} \mathfrak{h}_n\rightarrow \lim_{n} [ J_{n}(Y), \Omega Z]\cong [ J(Y), \Omega Z],
\end{equation*}
where $\mathfrak{h}$ is the desired \textit{Cohen group}. The Cohen representation can be functorially faithful due to a suitable choice of category and a beautiful characterization of the group $\mathfrak{h}$ has been proved by Cohen.
\begin{theorem}[\cite{Cohen95}]
When $\mathfrak{h}_1=\mathbb{Z}$ or $\mathbb{Z}/p^r$, 
\begin{equation*}
{\rm Ker}(\bar{p}_n:\mathfrak{h}_n\twoheadrightarrow \mathfrak{h}_{n-1}) \cong {\rm Lie}(n),
\end{equation*}
where ${\rm Lie}(n)$ is the $\mathfrak{h}_1$-submodule of $\big(V=\langle x_1,x_2, \ldots x_n\rangle_{\mathfrak{h}_1}\big)^{\otimes n}$ spanned by Lie elements 
\begin{equation*}
[[x_{\sigma(1)},x_{\sigma(2)}],\ldots, x_{\sigma(n)}]
\end{equation*}
for any $\sigma \in \Sigma_n$.
\end{theorem}

\subsection{Wu's approach: bi-$\Delta$-extension and skeleton filtration of bi-$\Delta$-groups}\label{Wuapp} (\cite{Wu})
\subsubsection{The machinery}\label{machin}
According to the construction in Section \ref{Cohensection}, for any $\Delta$-set (group) $\mathcal{S}=\{S_n, d_i\}_{n\geq 0}$ we may define a \textit{Cohen set (group)} for each $n$ by 
\begin{equation*}
\mathfrak{h}_n\mathcal{S}=\{x\in S_n~|~d_0x=d_1x=\cdots=d_nx\},
\end{equation*}
and also the \textit{total Cohen set (group)} by 
\begin{equation*}
\mathfrak{h}\mathcal{S}=\lim_{n}\{p_n=d_{i\vert}: \mathfrak{h}_n\mathcal{S}\rightarrow \mathfrak{h}_{n-1}\mathcal{S}\}.
\end{equation*}
We should notice that $\mathfrak{h}_{n+1}$ defined in Section \ref{Cohensection} is denoted by $\mathfrak{h}_n$ here. The constructions are functorial.
\begin{lemma}
Given a morphism of $\Delta$-sets (groups) 
\[e: \mathcal{S}\rightarrow \mathcal{T},\]
there is an induced filtration preserving morphism of total Cohen sets (groups) 
\[ \mathfrak{h}e: \mathfrak{h}\mathcal{S}\rightarrow \mathfrak{h}\mathcal{T}.\]
\end{lemma}
\begin{lemma}
The functor $\mathfrak{h}$ is left exact, i.e., given any short exact sequence of $\Delta$-groups
\[\{1\}\rightarrow \mathcal{H}\rightarrow \mathcal{G}\rightarrow \mathcal{K}\rightarrow \{1\},\] 
the induced sequence 
\[\{1\}\rightarrow \mathfrak{h}\mathcal{H}\rightarrow \mathfrak{h}\mathcal{G}\rightarrow \mathfrak{h}\mathcal{K}\] 
is exact.
\end{lemma}
In general $p_n$ may not be surjective (see Theorem $1.2.2$ of \cite{Wu}). However, it may be surjective if the objects involved have more structures, for instance, the bi-$\Delta$-group structure. 
\begin{definition}
A \textit{bi-$\Delta$-set (group)} $\mathcal{S}=\{S_n, d_j, d^i\}_{n\geq 0}$ is a $\Delta$- and co-$\Delta$-set (group) with $d_j$ and $d^i$ as the structural morphisms respectively such that 
\begin{equation*}
d_jd^i=\left\{\begin{array}{ll}
d^{i-1}d_j& j<i,\\
{\rm id}& j=i,\\
d^{i}d_{j-1}&j>i.
\end{array}
\right.
\end{equation*}
A \textit{weak bi-$\Delta$-group} $\mathcal{G}=\{G_n\}_{n\geq 0}$ is a bi-$\Delta$-set such that each $G_n$ is a group and all faces $d_j$ are group homomorphisms. 
\end{definition} 
\begin{proposition}[Proposition $1.2.1$ in \cite{Wu}]\label{progroup}
For any weak bi-$\Delta$-group $\mathcal{G}$ and each $n$, the map $p_n: \mathfrak{h}_n\mathcal{G}\rightarrow \mathfrak{h}_{n-1}\mathcal{G}$ is an epimorphism with kernel the Moore cycles group $\mathcal{Z}_n(\mathcal{G})$, where 
\[\mathcal{Z}_n(\mathcal{G})=\bigcap_{i=0}^{n} {\rm Ker}(d_i: G_n\rightarrow G_{n-1}).\]
\end{proposition}
\begin{definition}
An \textit{$n$-partial bi-$\Delta$-group} $\mathcal{G}=\{G_k\}_{0\leq k\leq n}$ is a finite sequence of groups with faces $d_j$ and cofaces $d^i$ satisfying all of the structural relations of a bi-$\Delta$-group up to dimension $n$.
\end{definition}
\begin{definition}
Given an $n$-partial bi-$\Delta$-group $\mathcal{G}=\{G_k\}_{0\leq k\leq n}$, the \textit{bi-$\Delta$-extension} of $\mathcal{G}$ is a bi-$\Delta$-group $\Phi_n \mathcal{G}$ with inclusion $\mathcal{G} \hookrightarrow \Phi_n \mathcal{G}$ such that the following universal property holds:
\begin{itemize}
\item[~] For any bi-$\Delta$-group $\mathcal{K}$ and any $n$-partial bi-$\Delta$-group morphism $f: \mathcal{G}\rightarrow \mathcal{K}$, there exists a unique morphism of bi-$\Delta$-groups $\tilde{f}: \Phi_n \mathcal{G}\rightarrow \mathcal{K}$ such that $\tilde{f}_{\vert \mathcal{G}}=f$.
\end{itemize}
The existence of the bi-$\Delta$-extension was guaranteed by the explicit construction of $\Phi_n\mathcal{G}$ in \cite{Wu}. Roughly speaking, we first construct an $(n+1)$-partial bi-$\Delta$-group $\Phi_n^{n+1}\mathcal{G}$ by defining the $(n+1)$-stage to be the $(n+2)$-fold self free product of $G_n$ module the coface relations. The faces are then induced by the usual projections as we did for free groups. We then can iterate the process to get a tower of partial bi-$\Delta$-groups
\begin{equation*}
\mathcal{G}=\Phi_n^{n}\mathcal{G}\subseteq \Phi_n^{n+1}\mathcal{G}\subseteq \Phi_n^{n+2}\mathcal{G}\subseteq \cdots,
\end{equation*}
and set $\Phi_n \mathcal{G}=\bigcup_{k=0}^{\infty}  \Phi_n^{n+k}\mathcal{G}$. It is then straightforward to check that $\Phi_n\mathcal{G}$ satisfies the universal property.
\end{definition}
\begin{definition}
Let $\mathcal{G}=\{G_n\}_{n\geq 0}$ be a bi-$\Delta$-group, the \textit{$n$-skeleton of $\mathcal{G}$} is defined to be 
\[{\rm sk}_n \mathcal{G}=\Phi_n {\rm Par}_n \mathcal{G},\]
where ${\rm Par}_n \mathcal{G}=\{G_k\}_{0\leq k\leq n}$ is an $n$-partial bi-$\Delta$-group with the induced faces and cofaces from $\mathcal{G}$.
\end{definition}
\begin{lemma}
Given any bi-$\Delta$-group morphism $f: \mathcal{G}\rightarrow \mathcal{K}$, 
we have a commutative diagram
\begin{equation*}
 \xymatrix{
   {\rm sk}_n \mathcal{G} 
        \ar[rr]^{{\rm sk}_n(f)} 
        \ar[d]
        &&
   {\rm sk}_n \mathcal{K}
      \ar[d] 
      \\
   \mathcal{G} 
     \ar[rr]^{f}
     &&
    \mathcal{K}. 
 }
\end{equation*}
\end{lemma}
\begin{proof}
The lemma follows from the diagram 
\begin{equation*}
 \xymatrix{
 {\rm Par}_n \mathcal{G} \ar[rr]_{\Phi_n} \ar[d]^{{\rm Par}_n(f)}
         \ar@{^{(}->}@/^1pc/[rrrr]&&
  {\rm sk}_n \mathcal{G} \ar@{.>}[rr] \ar@{.>}[d]^{{\rm sk}_n(f)} &&
   \mathcal{G}  \ar[d]^{f}\\
  {\rm Par}_n \mathcal{K} \ar[rr]_{\Phi_n} 
  \ar@{_{(}->}@/_1pc/[rrrr] &&
  {\rm sk}_n \mathcal{K} \ar@{.>}[rr] &&
  \mathcal{K},
 }
\end{equation*}
which is ensured by the universal property of $\Phi_n$.
\end{proof}
Given any group $G$, we may view it as a trivial $n$-partial bi-$\Delta$-group for any finite $n$ by defining the top stage to be $G$ and the remaining ones to be trivial groups. Then the only choices for the faces and cofaces  are the trivial morphisms, i.e., 
\[
\begin{tikzcd}
 \{1\} 
   \ar[r, shift left=1.2ex,"0"]
    \ar[r, shift left=0.4ex]
    &
\{1\} 
    \ar[l, shift left=0.4ex]
    \ar[l, shift left=1.2ex, "0"]
    \ar[r, shift left=2ex,"0"]
    \ar[r, shift left=1.2ex]
    \ar[r, shift left=0.4ex]
    &
\ldots
   \ar[l, shift left=2ex,"0"]
    \ar[l, shift left=1.2ex]
    \ar[l, shift left=0.4ex]
    \ar[r, shift left=2ex,"0"]
    \ar[r, shift left=1ex]
    \arrow{r}[description]{\vdots}
    &
\{1\} 
  \ar[l, shift left=2ex,"0"]
    \ar[l, shift left=1ex]
    \ar[r, shift left=2ex,"0"]
    \ar[r, shift left=1ex]
    \arrow{r}[description]{\vdots}
    &
G.
\ar[l, shift left=2ex,"0"]
    \ar[l, shift left=1ex]
\end{tikzcd}
\]
By abuse of notation, we may write $\Phi_n(G)$ to denote the bi-$\Delta$-extension of this trivial $n$-partial bi-$\Delta$-group. Now suppose we have a bi-$\Delta$-group $\mathcal{K}$ and a group homomorphism $\phi: G\rightarrow \mathcal{Z}_n\mathcal{K}$. Since $G$ is trivial as a partial bi-$\Delta$-group, the composition $G\rightarrow \mathcal{Z}_n\mathcal{K}\hookrightarrow \mathfrak{h}_n\mathcal{K}\hookrightarrow K_n$ determines a morphism of $n$-partial bi-$\Delta$-groups:
\[\phi: G\rightarrow {\rm Par}_n \mathcal{K}.\]
We then have a bi-$\Delta$-extension of $\phi$ by the diagram
\begin{equation*}
\xymatrix{
G \ar[rr]^{\phi} \ar[d] &&
{\rm Par}_n \mathcal{K} \ar@{_{(}->}[d] \ar@{^{(}->}[rrd] \\
\Phi_n G \ar[rr]^{\phi} &&
{\rm sk}_n \mathcal{K} \ar[rr]&&
\mathcal{K},
}
\end{equation*}
which induces a homomorphism of the total Cohen groups
\begin{equation*}
\mathfrak{h}\phi: \mathfrak{h}\Phi_n G\rightarrow  \mathfrak{h}{\rm sk}_n \mathcal{K} \rightarrow  \mathfrak{h}\mathcal{K},
\end{equation*}
such that $\mathfrak{h}_n\phi=i \circ\phi: G\rightarrow \mathcal{Z}_n\mathcal{K}\hookrightarrow \mathfrak{h}_n\mathcal{K}$.

There is another type of bi-$\Delta$-extension which concerns relations in Cohen groups. 
\begin{definition}\label{normalext0}
Suppose we have a bi-$\Delta$-group $\mathcal{G}$ and an inclusion of $n$-partial bi-$\Delta$-groups $i: \mathcal{H}\longrightarrow {\rm Par}_{n}\mathcal{G}$ with $\mathcal{H}_n\unlhd \mathcal{G}_n$, then the \textit{normal bi-$\Delta$-extension} of $\mathcal{H}$ in $\mathcal{G}$ is a bi-$\Delta$-group $\mathcal{N}_\mathcal{G}\mathcal{H}$ such that 
\[{\rm Par}_{n}\mathcal{N}_\mathcal{G}\mathcal{H}=\mathcal{H},\]
and also the following property holds:
\begin{itemize}
  \item[~] For any sub-bi-$\Delta$-groups $\mathcal{K}$ of $\mathcal{G}$ with ${\rm Par}_{n}\mathcal{K}=\mathcal{H}$ and $\mathcal{K}_j \unlhd \mathcal{G}_j$ for any $j\geq n$, we have an injection of bi-$\Delta$-groups $\mathcal{N}_\mathcal{G}\mathcal{H}\hookrightarrow \mathcal{K}$.
\end{itemize}
In particular, if $\mathcal{H}=H$ is the trivial $n$-partial bi-$\Delta$-subgroup of $\mathcal{G}$ and $H\unlhd \mathcal{G}_n$, we may call $\mathcal{N}_{\mathcal{G},n}H=\mathcal{N}_\mathcal{G}H$ the \textit{normal bi-$\Delta$-extension} of the group $H$.
\end{definition}
The explicit construction of such an extension can be described as follows. For such $\mathcal{H}$, we may define $\mathcal{N}_\mathcal{G}\mathcal{H}_{n+1}\in \mathcal{G}_{n+1}$ to be the normal closure generated by $d^ix$ for all $x\in \mathcal{H}_n$. It is easy to check that $\mathcal{H}\coprod \mathcal{N}_\mathcal{G}\mathcal{H}_{n+1}$ with the induced faces and cofaces from $\mathcal{G}$ is an $(n+1)$-partial bi-$\Delta$-group. Then we may iterate the process and get the required bi-$\Delta$-extension $\mathcal{N}_\mathcal{G}\mathcal{H}$.

Now suppose we have a morphism of bi-$\Delta$-groups $f: \mathcal{G}\rightarrow \mathcal{K}$ and a composition map of $n$-partial bi-$\Delta$-groups 
\[\mathcal{H}\stackrel{i}{\hookrightarrow} {\rm Par}_n\mathcal{G}\stackrel{{\rm Par}_n f}{\longrightarrow} {\rm Par}_n \mathcal{K},\]
which is trivial. If $\mathcal{H}_n$ is a normal subgroup of $\mathcal{G}_n$, then we have the extension morphism
\[\mathcal{N}_\mathcal{G}\mathcal{H}\stackrel{i}{\hookrightarrow}\mathcal{G}\stackrel{ f}{\longrightarrow} \mathcal{K},\]
which is also trivial.  Hence the induced homomorphism of Cohen groups
\[\mathfrak{h}\mathcal{N}_\mathcal{G}\mathcal{H}\stackrel{\mathfrak{h}i}{\hookrightarrow}\mathfrak{h}\mathcal{G}\stackrel{\mathfrak{h}f}{\longrightarrow} \mathfrak{h}\mathcal{K}\]
is also trivial. 

We may summarize a useful case in the following lemma:
\begin{lemma}\label{Cohenrelation}
Given a bi-$\Delta$-group $\mathcal{K}$ and a group $G$ with a group homomorphism $\phi: G\rightarrow \mathcal{Z}_n\mathcal{K}$. Suppose we also have an $m$-partial bi-$\Delta$-subgroup $\mathcal{H}$ of ${\rm Par}_m \Phi_n G$ such that $\mathcal{H}_m\unlhd \Phi_n G_m$ and the composition morphism 
\[\mathcal{H}\stackrel{i}{\hookrightarrow} {\rm Par}_m \Phi_n G \stackrel{{\rm Par}_m \phi}{\longrightarrow} {\rm Par}_m \mathcal{K}\]
is trivial, then the morphism of Cohen groups 
\[\mathfrak{h}\mathcal{N}_{\Phi_nG}\mathcal{H}\stackrel{\mathfrak{h}i}{\hookrightarrow}\mathfrak{h}\Phi_nG\stackrel{\mathfrak{h}\phi}{\longrightarrow} \mathfrak{h}\mathcal{K}\]
is trivial.

In particular, if we are given a trivial $m$-partial bi-$\Delta$-normal subgroup $H$ of $\Phi_n G$ such that 
\[H\stackrel{i}{\hookrightarrow}  \Phi_n G_m \stackrel{\phi_m}{\longrightarrow} \mathcal{K}_m\]
is trivial, then the morphism of Cohen groups 
\[\mathfrak{h}\mathcal{N}_{\Phi_nG, m}H\stackrel{\mathfrak{h}i}{\hookrightarrow}\mathfrak{h}\Phi_nG\stackrel{\mathfrak{h}\phi}{\longrightarrow} \mathfrak{h}\mathcal{K}\]
is trivial.
\end{lemma}

The case when $n=0$ is of special interest. Now the given data is a bi-$\Delta$-group $\mathcal{K}$ with a homomorphism $\phi: G\rightarrow \mathcal{K}_0$, and we have a homomorphism of Cohen groups $\mathfrak{h}\phi:  \mathfrak{h}\Phi_0 G\rightarrow \mathfrak{h}{\rm sk}_0 \mathcal{K} \rightarrow  \mathfrak{h}\mathcal{K}$. According to \cite{Wu}, the bi-$\Delta$-extension of $G$ can be explicitly described up to isomorphism. Indeed, we can define 
\begin{equation*}
(\Phi_0 G)_n=\coprod_{k=1}^{n+1} (G_k=G)
\end{equation*}
as the $(n+1)$-fold self free product of $G$. We may denote $\iota_k: G\hookrightarrow \coprod_{k=1}^{n+1}G_k$ as the inclusion of the $k$-th component. Then the faces $d_i: (\Phi_0 G)_{n}\rightarrow (\Phi_0 G)_{n-1}$ and coface morphisms $d^i: (\Phi_0 G)_{n-1}\rightarrow (\Phi_0 G)_{n}$ can be defined by group homomorphisms which are uniquely determined by the relations
\begin{equation*}
d_i \iota_j=\left\{\begin{array}{ll}
\iota_j& j\leq i,\\
0& j=i+1,\\
\iota_{j-1}&j\geq i+2.
\end{array}
\right.
\ \ \ \  
d^i \iota_j=\left\{\begin{array}{ll}
\iota_j& j\leq i,\\
\iota_{j+1}&j\geq i+1.
\end{array}
\right.
\end{equation*}
By construction, we can easily see that
\[ \iota_{j+1}=d^nd^{n-1}\cdots d^{j+2}d^{j+1}d^{j-1}\cdots d^1d^0 : G\rightarrow (\Phi_0 G)_n.
\]

We notice that when $G=\mathbb{Z}$, $(\Phi_0 G)=\{F_{n+1}\}_{n\geq 0}$ as bi-$\Delta$ groups which was discussed in Section \ref{Cohensection}.

\subsubsection{~} We now apply the above constructions to the bi-$\Delta$-group 
\[\Omega Z^{\ast}(Y)=\{[Y^{\times (n+1)}, \Omega Z]\}_{n\geq 0},\] 
the co-$\Delta$-structure of which is induced by maps
\begin{equation*}
d_i: Y^{\times (n+1)}\rightarrow Y^{\times n}, \ \ \ \ (y_1, y_2, \ldots, y_{n+1})\mapsto (y_1, \ldots, y_i, y_{i+2},\ldots, y_{n+1}). 
\end{equation*}

We start with a representation $\phi: G\rightarrow [Y^{\wedge (n+1)}, \Omega Z]=\mathcal{Z}_n\Omega Z^{\ast}(Y)$. Then according to the construction in Subsection \ref{machin}, we have a bi-$\Delta$-extension $\phi: \Phi_n G\rightarrow {\rm sk}_n \Omega Z^{\ast}(Y) \rightarrow \Omega Z^{\ast}(Y)$ and then a group homomorphism of Cohen groups 
\[\mathfrak{h}\phi:  \mathfrak{h}\Phi_n G\rightarrow \mathfrak{h}{\rm sk}_n \Omega Z^{\ast}(Y) \rightarrow  \mathfrak{h}\Omega Z^{\ast}(Y)\cong [J(Y), \Omega Z].\]
When $n=0$, ${\rm Im}(\phi_m)\subseteq [Y^{\times (m+1)}, \Omega Z]$ is the subgroup generated by 
\[ d^md^{m-1}\cdots d^{j+2}d^{j+1}d^{j-1}\cdots d^1d^0(\phi(G))=p_{j+1}^\ast (\phi(G)),
\]
where $p_{j+1}: Y^{\times (n+1)} \rightarrow Y$ is defined by $p_{j+1}(y_1,\ldots, y_{m+1})=y_{j+1}$ which exactly corresponds to $\iota_{j+1}$. If further $G=\mathbb{Z}$, the map $\phi$ is equivalent to a choice of map $f=\phi(1): Y\rightarrow \Omega Z$. Then the theory developed here covers Cohen's approach. Indeed, we may define 
\[H=\langle [[x_i, x_j], x_i],[[x_i, x_j], x_j]~|~1\leq i\neq j\leq 2 \rangle_N\unlhd (\Phi_0 \mathbb{Z})_1= F(x_1,x_2).\] 
Then the composition 
\[H\hookrightarrow (\Phi_0 \mathbb{Z})_1 \stackrel{\phi_2}{\rightarrow}\Omega Z^{\ast}(Y)_1=[Y^{\times 2}, \Omega Z]\]
is trivial if the reduced diagonal for $Y$ is null homotopic. Hence by Lemma \ref{Cohenrelation} the composition of homomorphisms
\[\mathfrak{h}\mathcal{N}_{\Phi_0 \mathbb{Z},1}H\hookrightarrow \mathfrak{h} \Phi_0 \mathbb{Z}\stackrel{\mathfrak{h}\phi}{\rightarrow} [J(Y), \Omega Z]\]
is trivial. For other cases we can proceed similarly.

 \section{Evaluation map and shuffle relations in Cohen groups}\label{Eva+Shsection}
\noindent Since our goal is to detect relations in the group of homotopy classes of self maps of double loop suspensions, we have to study the loop homomorphism 
\[\Omega: [\Omega\Sigma^2 X, \Omega Z]\rightarrow \lbrack \Omega^2\Sigma^2 X, \Omega^2 Z\rbrack.\]
In this section, we give a combinatorial model of the evaluation map ${\rm ev}: \Sigma\Omega^2\Sigma^2 X\rightarrow \Omega\Sigma^2 X$ as the adjoint of $\Omega( {\rm id})$. We then use this model to develop a general method to detect relations  in Cohen groups. As an illustration we construct the shuffle relations. 
\subsection{The evaluation map ${\rm ev}: \Sigma\Omega^2\Sigma^2 X\rightarrow \Omega\Sigma^2 X$}\label{evaluation}
 In \cite{Milgram74}, Milgram constructed his model for double loop suspensions with a different method. Explicitly, he defined a space of paths in $n$-cube $I^n$ by 
 \begin{equation*}
 {\rm Path}(n)=\big\{
  \sigma: [0,1]\rightarrow I^n~|~\sigma(0)=(0,\ldots,0), \sigma(1)=(1,\ldots,1).
  \big\}
 \end{equation*}
 Then for any connected space $X$, there is a map defined as the composition
 \begin{equation*}
  \psi_n: {\rm Path}(n) \times X^{\times n} \rightarrow \Omega((\Sigma X)^{\times n}) \rightarrow \Omega (J_n(\Sigma X)),
 \end{equation*}
 where $J_n$ is the $n$-th filtration of the classical James construction $J$, the first map is just a shuffle of variables and the second map is the loop of the natural projection. 
On the other hand, Milgram \cite{Milgram} also defined a family of maps 
\begin{equation*}
r_n^\prime: P_n \rightarrow {\rm Path}(n)
\end{equation*}
 by induction on $n$. Indeed, the map sends each vertex to a regular combinatorial path passing $n$ (dimension $1$) edges of $I^n$ in an order determined by the coordinate of the vertex. Any inner point $y$ of $P_n$ can be uniquely expressed as the linear combination of the center $c_n$ of $P_n$ and a point on the boundary, say $z=(z_i, z_j)\in P_i\times P_j$ with $i+j=n$ and $y=(1-t)c_n+tz$ for instance. By induction, we have defined $r_i^\prime(z_i)$ and $r_j^\prime(z_j)$ which determine a path $r_n^\prime(z)$ in $I^n$ through a canonical map 
 \[{\rm Path}(i)\times {\rm Path}(j)\rightarrow {\rm Path}(n).\] 
Then $r_n^\prime(y)$ is obtained by shrinking $r_n^\prime(z)$ to a path $t\cdot r_n^\prime(z)$ connecting $(0,\ldots, 0)$ and $(t,\ldots, t)$, and then joining $(t,\ldots, t)$ with $(1,\ldots ,1)$ by a line segment. It is easy to see this map can be equivariantly extended to $\mathcal{F}(n)$, and we have a map 
\begin{equation*}
r_n: \mathcal{F}(n) \rightarrow {\rm Path}(n),
\end{equation*} 
whose adjoint map $(r_n)^{\sharp}: I\times \mathcal{F}(n)\rightarrow I^n$ is cellular and of degree plus or minus one (Lemma $4.6$ in \cite{Milgram}). Combining the above together, we have constructed a map 
\begin{equation*}
\phi_n: \mathcal{F}(n)\times X^{\times n}\stackrel{r_n\times {\rm id}}{\longrightarrow}{\rm Path}(n)\times X^{\times n} \stackrel{\psi_n}{\longrightarrow} \Omega (J_n(\Sigma X)).
\end{equation*}
Then we can check that the map $\phi_n$ for each $n$ is compatible with the defining relations of $\mathcal{F}(X)$, and the diagrams
\begin{equation*}
    \xymatrix{  \mathcal{F}(n)\times X^{\times n} \ar[rr]^{\phi_n} \ar@{->>}[d]&&
                         \Omega(J_n(\Sigma X)) &\coprod_n \mathcal{F}(n)\times X^{\times n} \ar[rr]^{\coprod_n \Omega i_n\circ\phi_n} \ar@{->>}[d]&&
                         \Omega(J(\Sigma X))\\
                         \mathcal{F}_n(X) \ar@{.>}[urr]_{\tilde{\phi}_n}&&&\mathcal{F}(X) \ar@{.>}[urr]_{\tilde{\phi}}
  }
\end{equation*}
are commutative, where $i_n: J_n(\Sigma X) \rightarrow J(\Sigma X)$ is the natural inclusion. The theorem of Milgram (Corollary $0.20$ in \cite{Milgram74}, Theorem \ref{Milgram}) claims that the map $\tilde{\phi}$ is a homotopy equivalence. We can then choose the adjoint map of $\tilde{\phi}$ as our evaluation map 
\[{\rm ev}:=(\tilde{\phi})^\sharp: \Sigma \mathcal{F}(X) \rightarrow J(\Sigma X).\]
According to the construction, we see that $\tilde{\phi}$ preserves the natural filtrations and gives the following commutative diagram:
\begin{equation*}
  \xymatrix{
     \Sigma \mathcal{F}_{k-1}(X) \ar[rr]^{{\rm ev}_{k-1}} \ar@{_{(}->}[d] &&
      J_{k-1}(\Sigma X)  \ar@{_{(}->}[d]\\
    \Sigma \mathcal{F}_{k}(X) \ar[rr]^{{\rm ev}_{k}} \ar@{->>}[d] &&
      J_{k}(\Sigma X)  \ar@{->>}[d]\\
     \Sigma D_k(X) \ar[rr]^{\bar{{\rm ev}}_{k}} &&
     (\Sigma X) ^{\wedge k},
  }
\end{equation*}
where $D_k(X)=\mathcal{F}(k)^{+}\wedge_{\Sigma_k} X^{\wedge k}$ is the $k$-th divided power. Further, since the map $r_k$ sends any point on the boundary of $\mathcal{F}(k)$ to a path on the boundary of $I^k$, $\tilde{\phi}_n$ sends any point $[a, x]\in (\partial\mathcal{F}(n)\times X^{\times n}/\sim) \subset \mathcal{F}_{n}(X)$ to a loop in $\Omega J_{k-1}(\Sigma X)\subset \Omega J_{k}(\Sigma X)$. Hence, $\bar{{\rm ev}}_{k}$ can be factored as 
\begin{equation*}
\Sigma D_k(X)=\Sigma \mathcal{F}(k)^{+}\wedge_{\Sigma_k} X^{\wedge k} \stackrel{\Sigma q\wedge {\rm id}}{\longrightarrow} \Sigma \bigvee_{k!} S^{k-1} \wedge_{\Sigma_k} X^{\wedge k}\simeq  (\Sigma X) ^{\wedge k}\stackrel{e}{\longrightarrow}  (\Sigma X) ^{\wedge k},
\end{equation*}
where $q$ is the quotient map given by shrinking the $(k-2)$-skeleton of $\mathcal{F}(k)$ to a point. Also, since $(r_n)^{\sharp}$ is of degree plus or minus one and $\bar{{\rm ev}}_{k}$ is functorial in $X$, the map $e$ is of the form $i\times {\rm id}: S^k\wedge X^{\wedge n} \rightarrow  S^k\wedge X^{\wedge n}$ with ${\rm degree}~(i)=\pm 1$. In particular, $e$ is a homotopy equivalence, and we may choose the natural quotient $\tilde{q}_k=\Sigma q\wedge {\rm id}$ as the induced evaluation map $\bar{{\rm ev}}_{k}$.

\subsection{Shuffle map}
The shuffle map was defined in \cite{Wu} to construct shuffle relations in Cohen groups. Here, we recall the definition of the shuffle map and then give a combinatorial description using Milgram's model. 
\subsubsection{~} A convenient way to define the shuffle map is by working at the algebraic level first. Let $T(V)$ be the tensor coalgebra with $V$ the primitive base ($V$ will be realized as the homology of a co-$H$-space later), where the reduced coalgebra morphism is denoted by $\bar{\psi}: T(V)\rightarrow T(V)\otimes T(V)\twoheadrightarrow T(V)\wedge T(V)$. The morphisms $\bar{\psi}$ preserves the usual word length filtration. Let us denote 

\begin{equation*}
J_n(V)=\bigoplus_{j=0}^{n}T_j(V), \ \ T_n(V)=V^{\otimes n}.
\end{equation*}
Then we have a commutative diagram 
\begin{equation}\label{algsh}
  \xymatrix{
  J_{k-1}(V)\ar@{^{(}->}[r]  \ar[d]^{\bar{\psi}|_{ J_{k-1}}}&
  J_k (V)\ar@{->>}[r]  \ar[d]^{\bar{\psi}|_{ J_{k}}}&
  V^{\otimes k}\ar[d]^{{\rm sh}}\\
  {\rm Fil}_{k-1}(T(V)\wedge T(V))  \ar@{^{(}->}[r] &
  {\rm Fil}_{k}(T(V)\wedge T(V))  \ar@{->>}[r]&
  \bigoplus_{i,j >0,i+j=k} V^{\otimes i} \otimes V^{\otimes j}. 
  }
\end{equation}
The shuffle map can be defined by realizing the above diagram. Explicitly, suppose $Y$ is a co-$H$-space, then by definition we have a homotopy commutative diagram
\begin{equation*}
 \xymatrix{
 Y \ar[rr]^{\Delta} \ar[dr]_{\mu^\prime} &&
 Y\times Y\\
 & Y\vee Y\ar@{^{(}->}[ur]&,
 }
\end{equation*}
where $\mu^\prime$ is the comultiplication and $\Delta$ is the diagonal map. There is a homotopy commutative diagram 
\begin{equation*}
\begin{tikzcd}[column sep=0.17in, row sep=0.17in]
& 
(Y\vee Y)^{\times n}\ar[rrrr, dashrightarrow, "h"]\ar[ddd, twoheadrightarrow] \ar[dr, hookrightarrow]&&&&
\bigcup_n Y^{\times i}\wedge Y^{\times j}\ar[ddd, twoheadrightarrow]\ar[dl, hookrightarrow]\\
Y^{\times n} \ar[rr, crossing over, "\Delta^{\times n}~"] \ar[ru, "\mu^{\prime\times n}"] \ar[ddd, twoheadrightarrow] \ar[rrrr, crossing over, bend right=20, "\bar{\Delta}"] &&
(Y\times Y)^{\times n}\ar[rr, "\lambda"] \ar[ddd, crossing over, twoheadrightarrow]  &&
Y^{\times n}\wedge Y^{\times n}\ar[ddd, twoheadrightarrow]&\\
&&&&&\\
& 
J_n(Y\vee Y) \ar[rrrr, dashrightarrow, "\tilde{h}"] \ar[dr, hookrightarrow]&&&&
\bigcup_n J_i(Y)\wedge J_j(Y)\ar[dl, hookrightarrow]\\
J_n(Y)\ar[rr, "J_n(\Delta)"]  \ar[ru, "J_n(\mu^{\prime})"]  \ar[rrrr, bend right=20, "\bar{\Delta}"]&&
J_n(Y\times Y)\ar[rr, "\tilde{\lambda}"]&&
J_n(Y)\wedge J_n(Y)&,
\end{tikzcd}
\end{equation*}
where $\lambda ((y_1,y_1^\prime),\ldots, (y_n,y_n^\prime))=(y_1\wedge \ldots \wedge y_n)\wedge (y_1^\prime\wedge \ldots \wedge y_n^\prime)$, 
\begin{equation*}
\bigcup_n Y^{\times i}\wedge Y^{\times j}=\bigcup_{\substack{\{a_1,\ldots a_i, b_1, \ldots, b_j\}=\{1,\ldots, n\} \\i, j>0, ~a_1<\cdots< a_i, b_1<\cdots <b_j\\  }} (Y_{a_1}\times Y_{a_2}\times \cdots \times Y_{a_i})\wedge (Y_{b_1}\times Y_{b_2}\times \cdots \times Y_{a_j}),
\end{equation*}
$Y_{a_s}=Y$ presents the $a_s$-th $Y$ in the first $Y^{\times n}$ of $Y^{\times n}\wedge Y^{\times n}$ and $Y_{b_t}=Y$ presents the $b_t$-th $Y$ in the second $Y^{\times n}$. 
All the maps in the diagram are natural, and the composition map 
\begin{equation*}
\tilde{\Delta}: J_n(Y)\rightarrow J_n(Y\vee Y) \rightarrow \bigcup_n J_i(Y)\wedge J_j(Y)\end{equation*}
can be viewed as the reduced diagonal map. We then have a commutative diagram 
\begin{equation*}
  \xymatrix{
  J_{k-1}(Y)\ar@{^{(}->}[r]  \ar[d]^{\tilde{\Delta}}&
  J_k (Y)\ar@{->>}[r]  \ar[d]^{\tilde{\Delta}}&
 Y^{\wedge k}\ar[d]^{{\rm sh}}\\
  \bigcup_{k-1} J_i(Y)\wedge J_j(Y)  \ar@{^{(}->}[r] &
 \bigcup_{k} J_i(Y)\wedge J_j(Y) \ar@{->>}[r]&
  \bigvee_{i,j >0,i+j=k} Y^{\wedge i} \wedge Y^{\wedge j},
  }
\end{equation*}
which exactly realizes Diagram \ref{algsh}.
By straightforward computations, we may write the formula of the shuffle map as 
\begin{equation*}
{\rm sh}~(y_1\wedge \cdots \wedge y_k)=\sum_{\substack{i, j>0\\i+j=k}}\sum_{\substack{\{a_1,\ldots a_i, b_1, \ldots, b_j\}=\{1,\ldots, n\} \\ a_1<\cdots< a_i, b_1<\cdots <b_j\\  }}y_{a_1}\wedge \cdots \wedge y_{a_i}\wedge y_{b_1}\wedge \cdots \wedge y_{b_j},
\end{equation*}
which means the $(i, j)$-component of image of $y_1\wedge \cdots \wedge y_k$ in $Y^{\wedge i} \wedge Y^{\wedge j}$ under ${\rm sh}$ is 
\begin{equation*}
{\rm sh}_{i,j}(y_1\wedge \cdots \wedge y_k)=\sum_{\substack{\{a_1,\ldots a_i, b_1, \ldots, b_j\}=\{1,\ldots, n\} \\ a_1<\cdots< a_i, b_1<\cdots <b_j\\  }}y_{a_1}\wedge \cdots \wedge y_{a_i}\wedge y_{b_1}\wedge \cdots \wedge y_{b_j}.
\end{equation*}
Since the loop of the reduced diagonal is null homotopic, we have the following lemma which was used to detect the shuffle relations in Cohen groups \cite{Wu}: 
\begin{lemma}\label{oldnull}
The loop of the composite
\begin{equation*} 
J_k(Y)\twoheadrightarrow Y^{\wedge k} \stackrel{{\rm sh}}{\longrightarrow}\bigvee_{i,j >0,i+j=k} Y^{\wedge i} \wedge Y^{\wedge j} \stackrel{g}{\dashrightarrow} \Omega Z
\end{equation*}
is null-homotopic for any map $g$. 
\end{lemma}
\subsubsection{~} Now we will use Milgram's model to get a combinatorial description of the shuffle map. For the polyhedra $\mathcal{F}(k)$, we have the naive cofibration sequence 
\begin{eqnarray*}
\bigvee_{k!}S^{k-2}\stackrel{\beta}{\longrightarrow}{\rm sk}_{k-2} \mathcal{F}(k)/ {\rm sk}_{k-3} \mathcal{F}(k)& \hookrightarrow& \mathcal{F}(k)/ {\rm sk}_{k-3} \mathcal{F}(k) \twoheadrightarrow  \\
\mathcal{F}(k)/ {\rm sk}_{k-2} \mathcal{F}(k)\simeq \bigvee_{k!}S^{k-1} &\stackrel{\Sigma\beta}{\longrightarrow}&\Sigma {\rm sk}_{k-2}\mathcal{F}(k)/ {\rm sk}_{k-3} \mathcal{F}(k),  
\end{eqnarray*}
where $\beta$ is the attaching map. By applying $(-)^{+}\wedge_{\Sigma_k} X^{\wedge k}$ to the above sequence, we have a cofibration 
\begin{eqnarray*}
\Sigma^{k-2}X^{\wedge k}\stackrel{\beta^{+}\wedge_{\Sigma_k} X^{\wedge k}}{\longrightarrow}{\rm sk}_{k-2} D_k(X)/ {\rm sk}_{k-3} D_k(X)& \hookrightarrow& D_k(X)/ {\rm sk}_{k-3} D_k(X) \twoheadrightarrow  \\
D_k(X)/ {\rm sk}_{k-2} D_k(X)\simeq \Sigma^{k-1}X^{\wedge k} &\stackrel{\Sigma\tilde{\beta}}{\longrightarrow}&\Sigma{\rm sk}_{k-2} D_k(X)/ {\rm sk}_{k-3} D_k(X),  
\end{eqnarray*}
where ${\rm sk}_{i} D_k(X)={\rm sk}_{i}\mathcal{F}(k)^{+}\wedge_{\Sigma_k} X^{\wedge k}$ is the natural filtration, and 
\[\Sigma{\rm sk}_{k-2} D_k(X)/ {\rm sk}_{k-3} D_k(X) \simeq \Sigma \bigvee_{(k-1)\cdot k!} S^{k-2}\wedge_{\Sigma_k} X^{\wedge k}\simeq \bigvee_{k-1} \Sigma^{k-1}X^{\wedge k}.\]
\begin{lemma}\label{combish}
There is a homotopy commutative diagram 
\[
    \xymatrix{  
    (\Sigma X)^{\wedge k} \ar[rr]^{{\rm sh} \ \ \ \ \ \ \ \ \ \ \ } \ar[d]^{\simeq} &&
    \bigvee_{i,j>0, i+j=k} (\Sigma X)^{\wedge i}\wedge (\Sigma X)^{\wedge j} \ar[d]^{\simeq}\\
    \Sigma \bigvee_{k!} S^{k-1}\wedge_{\Sigma_k} X^{\wedge k} \ar[rr]^{\Sigma^2\tilde{\beta} \ \ \ \ \ }&&
    \Sigma^2 \bigvee_{(k-1)\cdot k!} S^{k-2}\wedge_{\Sigma_k} X^{\wedge k}.
  }
\]
\end{lemma}
\begin{proof}
In order to prove the lemma, we need to label the spheres involved by their corresponding elements in the poset of $\mathcal{F}(k)$. Hence, we may write 
\[\mathcal{F}(k)/ {\rm sk}_{k-2} \mathcal{F}(k)\simeq \bigvee_{\sigma \in \Sigma_k}S^{k-1}_{(\sigma)}=S^{k-1},\]
\[{\rm sk}_{k-2}\mathcal{F}(k)/ {\rm sk}_{k-3} \mathcal{F}(k) \simeq \bigvee_{(\tau, j)}S^{k-2}_{(\tau, j)}=S^{k-2},\]
where $(\tau, j)\in \mathcal{L}^{k-2}(k)$ (the set of elements of degree $k-2$ in $\mathcal{L}(k)$) corresponds to 
\[
  \xymatrix{
 \textbf{k} \ar[r]^{\tau} \ar@{->>}[d]^{\pi_j} &\textbf{k}\\
 \mathbf{\ubar{2}}
  }
\]
such that $\pi_j (i)=1$ for $i\leq j$ and $\pi_j (i)=2$ for $i> j$.
Since $\gamma^\ast(a=(\tilde{a}, \pi_a))= (\gamma^{-1}\circ \tilde{a},\pi_a)$ for any $\gamma \in \Sigma_k$ and $a\in \mathcal{L}(k)$, then in our two wedge of spheres
\[(\sigma^{-1})^\ast: S^{k-1}_{({\rm id})} \leadsto S^{k-1}_{(\sigma)},  \  \ 
 S^{k-2}_{({\rm id}, j)} \leadsto S^{k-2}_{(\sigma, j)}. \]
 On the other hand, we have a homotopy commutative diagram 
\[
 \xymatrix{
\Sigma S^{k-1}_{({\rm id})}\wedge X^{\wedge k}   
        \ar@{_{(}->}[d] \ar[rr]^{\beta^\prime=\Sigma^2\beta\wedge X^{\wedge k}_{\vert}} \ar@/_5pc/[dd]_{{\rm id}} &&
\Sigma^2 \bigvee_{(\tau, j)\in \mathcal{L}_{{\rm id}}^{k-2}(k)}S^{k-2}_{(\tau, j)} \wedge X^{\wedge k} 
        \ar@{_{(}->}[d]  \ar@/^6pc/[dd]^{\theta} \\
 \Sigma \bigvee_{\sigma \in \Sigma_k}S^{k-1}_{(\sigma)} \wedge X^{\wedge k} \ar@{->>}[d] \ar[rr]^{\Sigma^2\beta\wedge X^{\wedge k} } &&
\Sigma^2 \bigvee_{(\tau, j)\in \mathcal{L}^{k-2}(k)}S^{k-2}_{(\tau, j)}\wedge X^{\wedge k} \ar@{->>}[d]\\
\Sigma \bigvee_{\sigma \in \Sigma_k}S^{k-1}_{(\sigma)} \wedge_{\Sigma_k} X^{\wedge k}  \ar[rr]^{\Sigma^2\tilde{\beta}}&&
\Sigma^2 \bigvee_{(\tau, j)\in \mathcal{L}^{k-2}(k)}S^{k-2}_{(\tau, j)}\wedge_{\Sigma_k} X^{\wedge k}, 
 }
\]
where the composition 
\[\Sigma S^{k-1}_{({\rm id})}\wedge X^{\wedge k} \stackrel{\beta^\prime}{\rightarrow}
\Sigma^2 \bigvee_{(\tau, j)}S^{k-2}_{(\tau, j)} \wedge X^{\wedge k} \stackrel{p_{(\tau, j)}}{\rightarrow}S^{k-2}_{(\tau, j)} \wedge X^{\wedge k}\]
is homotopic to the identity for any projection $p_{(\tau, j)}$ of the ${(\tau, j)}$-component. 
The composition 
\[\Sigma^2 \bigvee_{j}S^{k-2}_{({\rm id}, j)}\wedge X^{\wedge k}\hookrightarrow \Sigma^2 \bigvee_{(\tau, j)\in \mathcal{L}^{k-2}(k)}S^{k-2}_{(\tau, j)}\wedge X^{\wedge k}\twoheadrightarrow\Sigma^2 \bigvee_{(\tau, j)\in \mathcal{L}^{k-2}(k)}S^{k-2}_{(\tau, j)}\wedge_{\Sigma_k} X^{\wedge k}\]
is also homotopic to the identity, and then the homotopy inverse of this map with the pre-composition of $\theta$ defines a map 
\[\tilde{\theta}:\Sigma^2 \bigvee_{(\tau, j)\in \mathcal{L}_{{\rm id}}^{k-2}(k)}S^{k-2}_{(\tau, j)} \wedge X^{\wedge k} \rightarrow  \Sigma^2 \bigvee_{j}S^{k-2}_{({\rm id}, j)}\wedge X^{\wedge k},\]
which essentially is the map 
\[\bigvee_{1\leq j\leq k-1}\big(\nabla\circ \bigvee_{(\tau, j)}\tau^{\ast}\big):\Sigma^2 \bigvee_{1\leq j\leq k-1}\bigvee_{(\tau, j)\in{\rm sh}(j, k-j)} S^{k-2}_{(\tau, j)} \wedge X^{\wedge k} \rightarrow  \Sigma^2 \bigvee_{1\leq j\leq k-1}S^{k-2}_{({\rm id}, j)}\wedge X^{\wedge k}\]
such that \[\tau^\ast(t_1\wedge \ldots\wedge t_k\wedge  x_1\wedge \ldots\wedge x_k)= (t_{\tau(1)}\wedge \ldots\wedge t_{\tau(k)}\wedge  x_{\tau(1)}\wedge \ldots\wedge x_{\tau(k)}),\]
and $\nabla$ is the folding map.
Hence under the equivalence $\Sigma^2 \bigvee_{j}S^{k-2}_{({\rm id}, j)}\wedge X^{\wedge k}\simeq \bigvee_j (\Sigma X)^{\wedge k}$ and other similar ones,
\begin{eqnarray*}
\tilde{\theta} \circ \beta^\prime 
&\simeq&
\big(\bigvee_{1\leq j\leq k-1}(\nabla\circ \bigvee_{(\tau, j)}\tau^{\ast})\big)\circ \beta^\prime \\
&\simeq&
\big(\bigvee_{1\leq j\leq k-1}(\nabla\circ \bigvee_{(\tau, j)}\tau^{\ast})\big)\circ (\bigvee_{1\leq j\leq k-1}u^\prime)\circ u^\prime\\
&\simeq& 
\bigvee_{1\leq j\leq k-1} \big(\nabla\circ(\bigvee_{(\tau, j)}\tau^{\ast})\circ u^\prime\big)\circ u^\prime  \\
&\simeq&
 \bigvee_{1\leq j\leq k-1}{\rm sh}_{j, k-j}\circ u^\prime\\
&\simeq&
{\rm sh},
\end{eqnarray*}
where $u^\prime$'s are the appropriate iterations of the co-multiplication of $(\Sigma X)^{\wedge k}$.
\end{proof}
According to this lemma, we may also denote the double suspended attaching map $\Sigma^2\tilde{\beta}$ by ${\rm sh}$. Combining the description of the evaluation map in Section \ref{evaluation} we have the following corollary:
\begin{corollary}[cf. Lemma \ref{oldnull}]\label{newnull0}
\[{\rm sh}\circ\bar{{\rm ev}}_k: \Sigma D_k(X)\twoheadrightarrow (\Sigma X)^{\wedge k}
\longrightarrow \bigvee_{k-1}(\Sigma X)^{\wedge k}\]
is null homotopic.
\end{corollary}

\subsection{The shuffle relations in Cohen groups}\label{Sh+Co}
In this subsection, we establish a canonical way to detect relations in Cohen groups using generalized normal bi-$\Delta$-extensions. We apply our method to derive shuffle relations as an example.
\subsubsection{Generalized normal bi-$\Delta$-extension}
Before discussing the shuffle relations, we need a more general concept of a normal bi-$\Delta$-extension (See Definition \ref{normalext0}).
\begin{definition}\label{normalext1}
Given a bi-$\Delta$-group $\mathcal{G}$ and a sequence of subgroups $\mathcal{H}=\{H_i~|~H_i\leq\mathcal{G}_i\}_{i\geq 0}$, the \textit{normal bi-$\Delta$-extension} of $\mathcal{H}$ in $\mathcal{G}$ is a bi-$\Delta$-group $\mathcal{N}_\mathcal{G}\mathcal{H}$ such that it is the minimal bi-$\Delta$-subgroup of $\mathcal{G}$ with the properties: 
\begin{itemize}
  \item[~] $\mathcal{N}_\mathcal{G}\mathcal{H}_n\unlhd \mathcal{G}_n$ and $H_n\leq\mathcal{N}_\mathcal{G}\mathcal{H}_n$ for each $n\geq 0$.
\end{itemize}
\end{definition}
The existence of such an extension can be confirmed by direct construction which can be obtained by induction. Indeed, for $H_0$ we have the usual normal bi-$\Delta$-extension $\mathcal{N}_\mathcal{G}\langle H_0\rangle_N$ which we may denote by $\mathcal{N}_\mathcal{G}^{(0)}\mathcal{H}$. Suppose we have constructed a sequence of bi-$\Delta$-groups 
\[\mathcal{N}_\mathcal{G}^{(0)}\mathcal{H}\subseteq\mathcal{N}_\mathcal{G}^{(1)}\mathcal{H}\subseteq\cdots \subseteq\mathcal{N}_\mathcal{G}^{(n-1)}\mathcal{H}\subseteq \mathcal{G}\]
such that $\mathcal{N}_\mathcal{G}^{(i)}\mathcal{H}$ contains $H_j$ for $j\leq i$ and the group in each stage is normal. Then we construct $\mathcal{N}_\mathcal{G}^{(n)}\mathcal{H}$ as follows. First, we may add $\langle H_n\rangle_N$ to the $n$-stage of $\mathcal{N}_\mathcal{G}^{(n-1)}\mathcal{H}$ and take group closure. Then add all the iterated face images of the elements of $H_n$ into the lower stages and take normal closures. We denote this new sequence of groups by $\mathcal{M}_\mathcal{G}^{(n)}\mathcal{H}$, and it is clearly that ${\rm Par}_n\mathcal{M}_\mathcal{G}^{(n)}\mathcal{H}$ is an $n$-partial $\Delta$-normal subgroup of $\mathcal{G}$. For any such partial $\Delta$-subgroup, we can make it into a partial bi-$\Delta$-subgroup by adding cofaces images from lower to higher stages. Further, by taking necessary normal closures, we may obtain an $n$-partial bi-$\Delta$-normal subgroup $\mathcal{P}_\mathcal{G}^{(n)}\mathcal{H}$ of $\mathcal{G}$. Then the usual normal bi-$\Delta$-extension $\mathcal{N}_\mathcal{G}\mathcal{P}_\mathcal{G}^{(n)}\mathcal{H}$ is the required $\mathcal{N}_\mathcal{G}^{(n)}\mathcal{H}$ and we complete the inductive step.
\begin{lemma}
The bi-$\Delta$-group $\mathcal{N}_\mathcal{G}\mathcal{H}=\bigcup_n\mathcal{N}_\mathcal{G}^{(n)}\mathcal{H}$ constructed above is the normal bi-$\Delta$-extension of $\mathcal{H}$ in $\mathcal{G}$.
\end{lemma}
\begin{proof}
The inductive step of the construction can be illustrated by the following diagram:
\[
\xymatrix{
  {\rm Par}_{n-1} \mathcal{N}_\mathcal{G}^{(n-1)}\mathcal{H}
  \ar@{^{(}->}[rr]^{\mathcal{N}_\mathcal{G}} 
  \ar@{_{(}->}[d] 
  &&
  \mathcal{N}_\mathcal{G}^{(n-1)}\mathcal{H}
  \ar@{_{(}->}[d]^{+(\langle H_n\rangle_N, d_j)}
 \\
  {\rm Par}_{n} \mathcal{M}_\mathcal{G}^{(n)}\mathcal{H}
 \ar@{^{(}->}[rr]
 \ar@{_{(}->}[d]^{+(d^i)}
 &&
 \mathcal{M}_\mathcal{G}^{(n)}\mathcal{H}
 \ar@{_{(}->}[d]
 \\
 \mathcal{P}_\mathcal{G}^{(n)}\mathcal{H}
 \ar@{^{(}->}[rr]^{\mathcal{N}_\mathcal{G}} 
 &&
 \mathcal{N}_\mathcal{G}^{(n)}\mathcal{H}.
}
\]
Hence, $\mathcal{N}_\mathcal{G}\mathcal{H}$ is a well defined bi-$\Delta$-normal subgroup of $\mathcal{G}$ containing each $H_n$ for $n\geq 0$. The minimality of $\mathcal{N}_\mathcal{G}\mathcal{H}$ is clear from the construction.
\end{proof}
Note that when $\mathcal{H}=\{H_i~|~H_i\unlhd\mathcal{G}_i\}_{0 \leq i\leq n}$ is an $n$-partial bi-$\Delta$-subgroup of $\mathcal{G}$, the two definitions of normal bi-$\Delta$-extensions give the same results. 

\subsubsection{~}\label{+Co} Now let us return to the homotopic context. 
For any $m\geq 1$, there is the subgroup ${\rm Sh}_{m}(\Sigma X; \Omega Z)$ of $\mathcal{Z}_m\Omega Z^{\ast}(\Sigma X)=[(\Sigma X)^{\wedge (m+1)}, \Omega Z]$ consisting of the homotopy classes of the maps of the form 
\[
(\Sigma X)^{\wedge (m+1)}\stackrel{{\rm sh}}{\rightarrow}\bigvee_{m}(\Sigma X)^{\wedge (m+1)}\stackrel{f}{\rightarrow}\Omega Z, \ \ \ \  {\rm for}~{\rm any}~ f: \bigvee_{m}(\Sigma X)^{\wedge (m+1)}\rightarrow \Omega Z.
\]
From Corollary \ref{newnull0}, we have that 
\[\bar{{\rm ev}}_{m+1}^\ast=0: {\rm Sh}^{N}_m(\Sigma X; \Omega Z) \longrightarrow \lbrack\Sigma D_{m+1}(X), \Omega Z\rbrack,\]
where ${\rm Sh}^{N}_m(\Sigma X; \Omega Z)$ is the normal closure of ${\rm Sh}_m(\Sigma X; \Omega Z)$ in $[(\Sigma X)^{\times (m+1)}, \Omega Z]$. (Note $[(\Sigma X)^{\times (m+1)}, \Omega Z]\cong [\Sigma (\Sigma X)^{\times (m+1)}, Z]$ may be non-abelian but admits another abelian group structure.)
Recall that we have the commutative diagram 
\[
  \xymatrix{
     \Sigma \mathcal{F}(X) \ar[rr]^{{\rm ev}}  &&
      J(\Sigma X) \\
    \Sigma \mathcal{F}_{m+1}(X) \ar[rr]^{{\rm ev}_{m+1}} \ar@{->>}[d]^{p}\ar@{^{(}->}[u] &&
      J_{m+1}(\Sigma X)  \ar@{->>}[d]^{p} \ar@{^{(}->}[u]\\
     \Sigma D_{m+1}(X) \ar[rr]^{\bar{{\rm ev}}_{m+1}} &&
     (\Sigma X) ^{\wedge (m+1)},
  }
\]
which implies that 
\[ {\rm ev}_{m+1}^\ast\circ p^\ast=0: {\rm Sh}_m^{N}(\Sigma X; \Omega Z) \longrightarrow \lbrack J_{m+1}(\Sigma X), \Omega Z\rbrack\longrightarrow \lbrack\Sigma \mathcal{F}_{m+1}(X), \Omega Z\rbrack.\]

Now suppose we have a representation $\phi: G\rightarrow \mathcal{Z}_n\Omega Z^{\ast}(\Sigma X)$. Then as before we have a bi-$\Delta$-extension $\phi: \Phi_n G\rightarrow \Omega Z^{\ast}(\Sigma X)$ and a group homomorphism of Cohen groups 
\[\mathfrak{h}\phi:  \mathfrak{h}\Phi_n G\rightarrow \mathfrak{h}\Omega Z^{\ast}(\Sigma X)\cong [J(\Sigma X), \Omega Z].\]
Meanwhile, we may define a sequence of normal subgroups ${\rm Sh}=\{{\rm Sh}_i\}_{i\geq 0}$ by the pullback diagram 
\[
\xymatrix{
{\rm Sh}_i \ar[rr] \ar@{_{(}->}[d] &&
{\rm Sh}_i^{N}(\Sigma X; \Omega Z) \ar@{_{(}->}[d] \\
\Phi_n G_i \ar[rr]^{\phi_i} &&
\Omega Z^{\ast}(\Sigma X)_i.
}
\]
Then apply the generalized normal bi-$\Delta$-extension, we have a commutative diagram of bi-$\Delta$-groups (${\rm Sh}^{N}(\Sigma X; \Omega Z)=\{{\rm Sh}^{N}_i(\Sigma X; \Omega Z) \}_{i\geq 0}$ and ${\rm Sh}^{N}_0(\Sigma X; \Omega Z)=\{0\}$)
\[
\xymatrix{
\mathcal{N}_{\Phi_n G}{\rm Sh} \ar[rr] \ar@{_{(}->}[d] &&
\mathcal{N}_{\Omega Z^{\ast}(\Sigma X)}{\rm Sh}^{N}(\Sigma X; \Omega Z)\ar@{_{(}->}[d] \\
\Phi_n G \ar[rr]^{\phi} &&
\Omega Z^{\ast}(\Sigma X),
}
\]
which implies a commutative diagram of Cohen groups 
\[
\xymatrix{
\mathfrak{h}\mathcal{N}_{\Phi_n G}{\rm Sh} \ar[rr] \ar@{_{(}->}[d] &&
\mathfrak{h}\mathcal{N}_{\Omega Z^{\ast}(\Sigma X)}{\rm Sh}^{N}(\Sigma X; \Omega Z) \ar@{_{(}->}[d] \\
\mathfrak{h}\Phi_n G \ar[rr]^{\mathfrak{h}\phi} &&
[J(\Sigma X), \Omega Z].
}
\]
Since ${\rm Sh}^{N}_m(\Sigma X; \Omega Z)\subseteq\mathcal{Z}_m\Omega Z^{\ast}(\Sigma X)$, any element in $\mathcal{N}_{\Omega Z^{\ast}(\Sigma X)}{\rm Sh}^{N}(\Sigma X; \Omega Z) $ is of the form $\prod_m d^{i_1}d^{i_2}\ldots d^{i_{k_m}}y_m$ with $y_m\in {\rm Sh}_m^{N}(\Sigma X; \Omega Z)$ and $i_1>i_2\cdots>i_m$, which in particular, implies that 
\[\mathcal{Z}_m(\mathcal{N}_{\Omega Z^{\ast}(\Sigma X)}{\rm Sh}^{N}(\Sigma X; \Omega Z) )={\rm Sh}_m^{N}(\Sigma X; \Omega Z).\] 
Then we see that 
\[\mathcal{Z}_m(\mathcal{N}_{\Phi_n G}{\rm Sh})={\rm Sh}_m.\]
Meanwhile at each $m$-th stage we have a commutative diagram
\[
\xymatrix{
\mathfrak{b}_m \mathcal{N}_{\Phi_n G}{\rm Sh} \ar[rr] \ar@{_{(}->}[d]^{b_m}&&
\mathfrak{b}_m\mathcal{N}_{\Omega Z^{\ast}(\Sigma X)}{\rm Sh}\ar@{_{(}->}[d]\\
\mathfrak{h}_m\mathcal{N}_{\Phi_n G}{\rm Sh} \ar[rr] \ar@{_{(}->}[d]^{i_m} &&
\mathfrak{h}_m\mathcal{N}_{\Omega Z^{\ast}(\Sigma X)}{\rm Sh}^{N}(\Sigma X; \Omega Z) \ar@{_{(}->}[d]^{i_m} \\
\mathfrak{h}_m\Phi_n G \ar[rr]^{\mathfrak{h}_m\phi}  \ar[drr]_{e_m}&&
[J_{m+1}(\Sigma X), \Omega Z] \ar[d]^{{\rm ev}_{m+1}^\ast}\\
 &&
\lbrack\Sigma \mathcal{F}_{m+1}(X), \Omega Z\rbrack,
}
\]
where $\mathfrak{b}_m\mathcal{N}_{\Omega Z^{\ast}(\Sigma X)}{\rm Sh}={\rm Ker}({\rm ev}_{m+1}^\ast\circ i_m)$ and the top square is defined by taking the pullback. 
We then have a commutative diagram
\[
\xymatrix{
{\rm Sh}_m^{N}(\Sigma X; \Omega Z) 
\ar@{^{(}->}[rr]  \ar@{=}[d]&&
\mathfrak{b}_{m}\mathcal{N}_{\Omega Z^{\ast}(\Sigma X)}{\rm Sh}
\ar[rr]^{p_m}
\ar@{_{(}->}[d]^{b_m}
&&
\mathfrak{b}_{m-1}\mathcal{N}_{\Omega Z^{\ast}(\Sigma X)}{\rm Sh}
\ar@{_{(}->}[d]^{b_{m-1}}
\\
{\rm Sh}_m^{N}(\Sigma X; \Omega Z) 
\ar@{^{(}->}[rr]  \ar@{_{(}->}[d]
&&
\mathfrak{h}_{m}\mathcal{N}_{\Omega Z^{\ast}(\Sigma X)}{\rm Sh}
\ar@{->>}[rr]^{p_m}
\ar@{_{(}->}[d]^{i_m}
&&
\mathfrak{h}_{m-1}\mathcal{N}_{\Omega Z^{\ast}(\Sigma X)}{\rm Sh}
\ar@{_{(}->}[d]^{i_{m-1}}
\\
[(\Sigma X)^{\wedge (m+1)}, \Omega Z]
\ar@{^{(}->}[rr]  \ar[d]^{\bar{{\rm ev}}_{m+1}^\ast}
&&
[J_{m+1}(\Sigma X), \Omega Z]
\ar@{->>}[rr]^{p_m}
 \ar[d]^{{\rm ev}_{m+1}^\ast}
 &&
 [J_{m}(\Sigma X), \Omega Z]
  \ar[d]^{{\rm ev}_{m}^\ast}
 \\
 [\Sigma D_{m+1}(X), \Omega Z]
 \ar[rr]
 &&
[\Sigma \mathcal{F}_{m+1}(X), \Omega Z]\ar[rr]
&&
[\Sigma \mathcal{F}_{m}(X), \Omega Z],
}
\]
where the composition of maps in each column is zero, and the kernel of $p_m$ in the top row is ${\rm Sh}_m^{N}(\Sigma X; \Omega Z)$. Then we have a morphism 
\[p_m: \mathfrak{b}_m \mathcal{N}_{\Phi_n G}{\rm Sh}\rightarrow \mathfrak{b}_{m-1} \mathcal{N}_{\Phi_n G}{\rm Sh}\]
with kernel ${\rm Sh}_m$, and denote $\mathfrak{b} \mathcal{N}_{\Phi_n G}{\rm Sh}=\lim_m\mathfrak{b}_m \mathcal{N}_{\Phi_n G}{\rm Sh}$.
\begin{proposition}\label{shthm}
We have a commutative diagram of groups
\[
\xymatrix{
\mathfrak{h}\Phi_n G 
\ar[rr]^{\mathfrak{h}\phi}
\ar@{->>}[d]
&&
[J(\Sigma X), \Omega Z]
\ar[d]^{\Omega}
\\
\mathfrak{h}\Phi_n G /\mathfrak{b}\mathcal{N}_{\Phi_n G}{\rm Sh}
\ar@{.>}[rr]^{\Omega\mathfrak{h}\phi}
&&
\lbrack \Omega^2\Sigma^2 X, \Omega^2 Z\rbrack.
}
\]
\end{proposition}
\begin{proof}
We have $e_{m}\circ i_m\circ b_m=0$ for each $m$ which implies at the infinity
\[e\circ i\circ b=0:\mathfrak{b}\mathcal{N}_{\Phi_n G}{\rm Sh}\rightarrow \mathfrak{h}\Phi_n G \stackrel{\mathfrak{h}\phi}{\rightarrow} [J(\Sigma X), \Omega Z]\stackrel{{\rm ev}^\ast}{\rightarrow}\lbrack\Sigma \mathcal{F}(X), \Omega Z\rbrack.\]
\end{proof}

\section{Higher relations in Cohen groups}\label{Todasection}
\noindent In the spirit of our combinatorial treatment in Section \ref{Sh+Co}, we may find more relations in Cohen groups for double loop suspensions using the cell structure of permutohedra.
\subsection{Ladder spaces and Toda brackets}
\begin{definition}
A \textit{(pointed) ladder space} is a sequence of pairs of (pointed) topological spaces $\{(A_i, B_i)\}_{0\leq i\leq n}$ ($n$ can be $\infty$) with structural (pointed) maps $f_i: B_i \rightarrow A_i$ such that 
\[B_i\stackrel{f_i}{\longrightarrow} A_i\longrightarrow A_{i+1}\]
is a co-fibre sequence for each $i$. We may also call the total space $A={\rm colim}_{i}A_i$ the ladder space, and denote $A=\{(A_i, B_i)\}_{i\geq 0}$.
\end{definition}
For any such ladder space, we have an organised diagram 
\[
\xymatrix{ 
B_0 \ar[r]^{f_0} &
A_0\ar[d] &&&&&
\\
B_1 \ar[r]^{f_1} &
A_1  \ar[d]\ar[r]^{p_1}&
\Sigma B_0 \ar[r]^{\Sigma f_0} &
\Sigma A_0 \ar[d] &
&&\\
B_2 \ar[r]^{f_2} &
A_2  \ar[d]\ar[r]^{p_2}&
\Sigma B_1 \ar[r]^{\Sigma f_1} &
\Sigma A_1 \ar[r]^{\Sigma p_1}\ar[d] &
\Sigma^2 B_0 \ar[r]^{\Sigma^2 f_0}&
\Sigma^2 A_0 \ar[d]
&\\
B_3 \ar[r]^{f_3} &
A_3  \ar[d]\ar[r]^{p_3}&
\Sigma B_2 \ar[r]^{\Sigma f_2} &
\Sigma A_2 \ar[d]\ar[r]^{\Sigma p_2}&
\Sigma^2 B_1 \ar[r]^{\Sigma^2 f_1}&
\Sigma^2 A_1 \ar[d]\ar[r]^{\Sigma^2 p_1}&
\Sigma^3 B_0 \ar[r]&
\cdots
\\
& \vdots &&
\vdots &&
\vdots &
}
\]
Denote $p_i\circ f_i=g_i$ for each $i\geq 1$. Then by definition $\Sigma g_{i}\circ g_{i+1}$ is null homotopic. The constructions here obviously lead us to the Toda bracket \cite{Toda}. For any sequence of based maps
\[X\stackrel{\alpha}{\rightarrow} Y\stackrel{\beta}{\rightarrow}Z\stackrel{\gamma}{\rightarrow}W\]
such that $\beta\circ\alpha$ and $\gamma\circ\beta$ are both null-homotopic, Toda defined a set \[\langle\gamma, \beta, \alpha\rangle_T\subseteq \lbrack\Sigma X, W\rbrack,\]
which is a certain double coset of $\gamma_\ast [\Sigma X, Z]$ and $(\Sigma \alpha)^\ast[\Sigma Y, W]$. The higher oder generalizations were defined and studied by Spanier \cite{Spanier} and Cohen \cite{JCohen}, and generalized by Shipley for any triangulated category in the new century \cite{Shipley}. For any sequence of based maps
\begin{equation}\label{nseq}
Y_0\stackrel{\alpha_1}{\rightarrow} Y_1\stackrel{\alpha_2}{\rightarrow}\cdots \stackrel{\alpha_n}{\rightarrow}Y_n,
\end{equation}
the classical higher Toda bracket
\[\langle \alpha_n, \cdots, \alpha_2, \alpha_1\rangle_{T}\]
is a certain subset of $[\Sigma^{n-2} Y_0 , Y_n]$. The existence and vanishing of higher Toda brackets are characterized by the following theorem of Spanier:
\begin{theorem}[Theorem $6.3$ of \cite{Spanier}]
The Toda bracket 
\[\langle \alpha_n, \cdots, \alpha_2, \alpha_1\rangle_{T}\]
is well defined and vanishes (which means `$0$' belongs to the set) if and only if the sequence (\ref{nseq}) \textit{splits} in the sense of the following:
\begin{itemize}
\item[~] There exists a commutative diagram 
\[
\xymatrix{
&&Z_2 \ar[d]^{\gamma_2} & Z_3\ar[d]^{\gamma_3}&\cdots&Z_{n-1}\ar[d]^{\gamma_{n-1}}\\
Y_0 \ar[r]_{\alpha_1} & Y_1\ar[r]_{\alpha_2}\ar[ru]_{\beta_2} &
Y_2 \ar[r]_{\alpha_3}\ar[ru]_{\beta_3}&
Y_4  \ar[r]_{\alpha_4}\ar[ru]&
\cdots \ar[ru]\ar[r]_{\alpha_{n-1}}&
Y_{n-1}\ar[r]_{\alpha_{n}}&
Y_{n},
}
\]
such that $\beta_{i+1}\circ \gamma_{i}$ for $2\leq i\leq n-2$ and the two end composites $\beta_2\circ \alpha_1$ and $\alpha_n\circ \gamma_{n-1}$ are null homotoptic.  
\end{itemize}
\end{theorem}
\begin{corollary}\label{laddertoda}
For any ladder space $A=\{(A_i, B_i)\}_{i\geq 0}$, the associated sequence 
\[B_{n}\stackrel{g_n}{\longrightarrow} \Sigma B_{n-1}\stackrel{\Sigma g_{n-1}}{\longrightarrow}\cdots \stackrel{\Sigma^{n-1}g_1}{\longrightarrow}\Sigma ^{n}B_{0}\]
gives an $n$-length Toda bracket
\[\langle \Sigma^{n-1}g_1, \cdots, \Sigma g_{n-1}, g_n\rangle_{T}\]
which is well defined and vanishes. 
\end{corollary}

\subsection{Secondary relations in Cohen groups}
Now for $D_k(X)=\mathcal{F}(k)^{+}\wedge_{\Sigma_k} X^{\wedge k}$, there are natural cofibrations 
\[\bigvee_{{k-1\choose{i+1}}}\Sigma^iX^{\wedge k}\simeq \bigvee_{k!\cdot {k-1\choose{k-i-2}}}S^{i}\wedge_{\Sigma_k} X^{\wedge k}\stackrel{f_i}{\longrightarrow}{\rm sk}_{i} D_k(X)\longrightarrow {\rm sk}_{i+1} D_k(X) \]
determined by the attaching maps of $\mathcal{F}(k)$ for any $0\leq i\leq k-2$, where ${\rm sk}_{i} D_k(X)={\rm sk}_{i}\mathcal{F}(k)^{+}\wedge_{\Sigma_k} X^{\wedge k}$.
Hence $D_k(X)$ is a ladder space such that 
\[D_k(X)=\{{\rm sk}_{i} D_k(X), \bigvee_{{k-1\choose{i+1}}}\Sigma^iX^{\wedge k}\}_{0\leq i\leq k-2}.\]
However, this ladder space is not useful as we will see in a moment. Instead, we consider the suspension of $D_k(X)$ with the inherited ladder structure:
\[\Sigma D_k(X)=\{\Sigma{\rm sk}_{i} D_k(X), \bigvee_{{k-1\choose{i+1}}}\Sigma^{i+1}X^{\wedge k}\}_{0\leq i\leq k-2}.\]
We then have a sequence of maps 
\begin{displaymath}
\xymatrix{
&\Sigma {\rm sk}_{k-2}D_k(X) \ar[d]^{\Sigma p_{k-2}}&
\Sigma^2 {\rm sk}_{k-3}D_k(X) \ar[d]^{\Sigma^2 p_{k-3}}&
\Sigma^3 {\rm sk}_{k-4}D_k(X) \ar[d]^{\Sigma^3 p_{k-4}}\\
\Sigma^{k-1}X^{\wedge k} \ar[ur]^{\Sigma f_{k-2}} \ar[r]_{\Sigma g_{k-2}}&
\Sigma^2\displaystyle{\bigvee_{ (k-1)}}\Sigma^{k-3}X^{\wedge k} \ar[ur]^{\Sigma^2 f_{k-3}} \ar[r]_{\Sigma^2 g_{k-3}}&
\Sigma^3 \displaystyle{\bigvee_{{k-1\choose{2}}}} \Sigma^{k-4}X^{\wedge k} \ar[ur]^{\Sigma^3 f_{k-4}} \ar[r]_{\Sigma^3 g_{k-4}}&
\Sigma^4 \displaystyle{\bigvee_{{k-1\choose{3}}}} \Sigma^{k-5}X^{\wedge k}
}
\end{displaymath}
which determines a Toda bracket $\langle \Sigma^3 g_{k-4},\Sigma^2 g_{k-3},\Sigma g_{k-2}\rangle_{T}$. By Corollary \ref{laddertoda}, we see this bracket vanishes. Hence $\langle \Sigma^3 g_{k-4},\Sigma^2 g_{k-3},\Sigma g_{k-2}\rangle_{T}$ as a group is isomorphic to the subgroup ${\rm Indet}_k={\rm Indet}((\Sigma X)^{\wedge k},\bigvee_{{k-1\choose{3}}}\Sigma^{-1}(\Sigma X)^{\wedge k})=$
\[(\Sigma^2 g_{k-2})^\ast[\bigvee_{k-1}(\Sigma X)^{\wedge k}, \bigvee_{{k-1\choose{3}}}\Sigma^{-1}(\Sigma X)^{\wedge k}]+(\Sigma^3g_{k-4})_\ast[(\Sigma X)^{\wedge k}, \bigvee_{{k-1\choose{2}}}\Sigma^{-1}(\Sigma X)^{\wedge k}]\]
of $[(\Sigma X)^{\wedge k},\bigvee_{{k-1\choose{3}}}\Sigma^{-1}(\Sigma X)^{\wedge k}]$.
The following lemma allows us to detect more relations in Cohen groups for double loop suspensions.
\begin{lemma}\label{2ndnull}
For any $h\in {\rm Indet}_k$, the composite $ h\circ \bar{{\rm ev}}_k$ is null homotopic. 
\end{lemma}
\begin{proof}
We have that $\Sigma g_{k-2}$ is the desuspension of the shuffle map, and also the composition of maps 
\[D_k(X) \stackrel{p_{k-1}=\Sigma^{-1}\bar{{\rm ev}}_{k}}{\longrightarrow}\Sigma^{k-1}X^{\wedge k} \stackrel{\Sigma g_{k-2}=\Sigma^{-1}{\rm sh}}{\longrightarrow}\displaystyle{\bigvee_{k-1}}\Sigma^{k-1}X^{\wedge k}\]
is null homotopic. On the other hand, Toda (Proposition $1.2$ of \cite{Toda}) proved that
\[\langle\gamma, \beta, \alpha\rangle_T \circ \Sigma \delta\subseteq \langle\gamma, \beta, \alpha\circ \delta\rangle_T,\] 
and \[\langle\gamma, \beta, \alpha\rangle_T=\{0\}\]
if one of $\alpha$, $\beta$ or $\gamma$ is $0$.
Accordingly, 
\begin{eqnarray*}
&&\langle \Sigma^3 g_{k-4},\Sigma^2 g_{k-3},\Sigma g_{k-2}\rangle_{T}\circ \bar{{\rm ev}}_k \\
&\subseteq& \langle \Sigma^3 g_{k-4},\Sigma^2 g_{k-3},\Sigma g_{k-2}\circ\Sigma^{-1} \bar{{\rm ev}}_k \rangle_{T}\\
&=&\{0\},
\end{eqnarray*}
and the lemma follows.
\end{proof}
Now we may define ${\rm T}_{m}(\Sigma X; \Omega Z)$ to be the subgroup 
\[ \lbrack (\Sigma X)^{\wedge (m+1)}\stackrel{h}{\rightarrow}\bigvee_{{m\choose{3}}}\Sigma^{-1}(\Sigma X)^{\wedge (m+1)}\stackrel{f}{\rightarrow}\Omega Z~|~\forall~f, \forall ~h\in {\rm Indet}_{m+1}\rbrack \]
of $\mathcal{Z}_m \Omega Z^\ast(\Sigma X)$, and go through all the arguments in Section \ref{+Co} using ${\rm T}_{m}(\Sigma X; \Omega Z)$ instead of ${\rm Sh}_m(\Sigma X; \Omega Z)$. We then summarize the corresponding results in the following theorem. 
\begin{theorem}\label{Todarelation}
Given a representation $\phi: G\rightarrow \mathcal{Z}_n\Omega Z^{\ast}(\Sigma X)$ and hence its bi-$\Delta$-extension $\phi: \Phi_n G\rightarrow\Omega Z^{\ast}(\Sigma X)$, we have a bi-$\Delta$-normal subgroup $\mathcal{N}_{\Phi_n G} {\rm T}$ of $\Phi_n G$ as the normal bi-$\Delta$-extension of ${\rm T}=\{{\rm T}_i\}_{i\geq 4}$, where $T_i$ is defined by the pullback diagram 
\[
\xymatrix{
{\rm T}_i \ar[rr] \ar@{_{(}->}[d] &&
{\rm T}_i^{N}(\Sigma X; \Omega Z) \ar@{_{(}->}[d] \\
\Phi_n G_i \ar[rr]^{\phi_i} &&
\Omega Z^{\ast}(\Sigma X)_i.
}
\]
Moreover, we have a commutative diagram of groups
\[
\xymatrix{
\mathfrak{h}\Phi_n G 
\ar[rr]^{\mathfrak{h}\phi}
\ar@{->>}[d]
&&
[J(\Sigma X), \Omega Z]
\ar[d]^{\Omega}
\\
\mathfrak{h}\Phi_n G /\langle \mathfrak{b}\mathcal{N}_{\Phi_n G}{\rm Sh},\mathfrak{b}\mathcal{N}_{\Phi_n G}{\rm T}\rangle_{N}
\ar@{.>}[rr]^{\Omega\mathfrak{h}\phi}
&&
\lbrack \Omega^2\Sigma^2 X, \Omega^2 Z\rbrack,
}
\]
where \[\mathfrak{b}\mathcal{N}_{\Phi_n G}{\rm T}=\lim_m\mathfrak{b}_m\mathcal{N}_{\Phi_n G}{\rm T},\] 
and $\mathfrak{b}_m\mathcal{N}_{\Phi_n G}{\rm T}$ is defined by the pullback diagram 
\[
\xymatrix{
\mathfrak{b}_m \mathcal{N}_{\Phi_n G}{\rm T} \ar[rr] \ar@{_{(}->}[d]^{b_m}&&
\mathfrak{b}_m\mathcal{N}_{\Omega Z^{\ast}(\Sigma X)}{\rm T}\ar@{_{(}->}[d]\\
\mathfrak{h}_m\mathcal{N}_{\Phi_n G}{\rm T} \ar[rr]&&
\mathfrak{h}_m\mathcal{N}_{\Omega Z^{\ast}(\Sigma X)}{\rm T}^{N}(\Sigma X; \Omega Z)
}
\]
with
\[\mathfrak{b}_m\mathcal{N}_{\Omega Z^{\ast}(\Sigma X)}{\rm T}={\rm Ker}(\mathfrak{h}_m\mathcal{N}_{\Omega Z^{\ast}(\Sigma X)}{\rm T}^{N}(\Sigma X; \Omega Z) \rightarrow [\Sigma\mathcal{F}_{m+1}(X), \Omega Z]).\]
In addition, the sequence
\[0\rightarrow T_m\rightarrow \mathfrak{b}_m\mathcal{N}_{\Phi_n G}{\rm T}\rightarrow \mathfrak{b}_{m-1}\mathcal{N}_{\Phi_n G}{\rm T}\]
is exact.
\end{theorem}

\section{Appendix: James-Hopf operations of abelian bi-$\Delta$-groups}\label{AppendA}
\noindent As we mentioned in the introduction, a Cohen representation can be used to study many kinds of maps in homotopy theory. One of the classic tools to construct maps is the James-Hopf invariant. In \cite{Wu}, the second author defined a combinatorial analogy to James-Hopf operation for any (weak) bi-$\Delta$-group. Given any bi-$\Delta$-group $\mathcal{G}$, the \textit{James-Hopf homomorphism} $H_{k,n}: \mathcal{G}_k\rightarrow \mathcal{G}_n$ is defined by $H_{k,k}={\rm id}$, and for $n\geq k$ and $x\in \mathcal{G}_k$
\[H_{k,n}(x)=\prod_{0\leq i_1<i_2<\cdots< i_{n-k}\leq n} d^{i_{n-k}}d^{i_{n-k-1}}\cdots d^{i_{1}}(x)\]
with lexicographic order from right. We may also set $H_{k,n}=0$ for $k>n$.
\begin{lemma}\label{chainhom}
Suppose $\mathcal{G}$ is an abelian bi-$\Delta$-group, then the James-Hopf operations induce a sequence of group homomorphisms of Cohen groups 
\[H_{k,n}:\mathfrak{h}_k\mathcal{G}\rightarrow \mathfrak{h}_n\mathcal{G}.\]
Moreover, there are combinatorial equalities
\[p_nH_{k,n}=H_{k,n-1}+H_{k-1,n-1}p_k,\]
\[H_{n,m}H_{k,n}={{m-k}\choose{m-n}} H_{k,m}.\]
And the $H_{k,n}$'s determine a well defined morphism (also in non-abelian case)
\[H_k: \mathcal{Z}_k\mathcal{G}\rightarrow \mathfrak{h}\mathcal{G}.\]
\end{lemma}

\begin{proposition}\label{split}
Given any abelian bi-$\Delta$-group $\mathcal{G}$ such that $\mathfrak{h}_i\mathcal{G}=0$ for any $i\leq k$, and $\mathcal{G}$ is $p$-local for some prime $p$ (which means each $\mathcal{G}_j$ is $p$-local),
then the short exact sequence
\[0\rightarrow \mathcal{Z}_m\mathcal{G}\rightarrow \mathfrak{h}_m\mathcal{G}\rightarrow \mathfrak{h}_{m-1}\mathcal{G}\rightarrow 0\]splits for any $m\leq p+k$.
Further, there is a short exact sequence for each $m$
\[0\rightarrow \mathcal{Z}_{m+1}\mathcal{G}\oplus\mathcal{Z}_{m+2}\mathcal{G}\oplus \cdots\oplus \mathcal{Z}_{m+p}\mathcal{G}\rightarrow \mathfrak{h}_{m+p}\mathcal{G}\rightarrow \mathfrak{h}_m\mathcal{G}\rightarrow 0.\]
\end{proposition}
\begin{proof} We only need to prove the second statement and for simplicity we may omit $\mathcal{G}$ in the notation.
By Lemma \ref{chainhom}, we have 
\[
p_n\sigma_n={\rm id}+\sigma_{n-1}p_{n-1},
\]
\[
H_{n,m}\sigma_n=(m-n+1)H_{n-1,m},
\]
where $\sigma_n=H_{n-1, n}=\sum\limits_{i=0}^{n}d^i$.
Hence, \[H_{n,m}\sigma_{n}\sigma_{n-1}\cdots\sigma_{s+1}=P_{n-s}^{m-s}H_{s,m},\]
where the combinatorial number 
\[P_{n-s}^{m-s}=(m-s)\cdots(m-n+2)(m-n+1).\]
In particular,
\[\sigma_{n}\sigma_{n-1}\cdots\sigma_{s+1}=(n-s)!\cdot H_{s,n}.\]
We then can show the following claims by straightforward calculations:
\begin{itemize}
\item 
\textbf{Claim $1$:}
~$p_{m+n}\sigma_{m+n}=n\cdot {\rm id}$ on $\mathcal{Z}_{m}$, where we view 
\[ \mathcal{Z}_{m}~{\rm as}~\sigma_{m+n-1}\sigma_{m+n-2}\cdots\sigma_{m+1}\mathcal{Z}_{m} \in \mathfrak{h}_{m+n-1}.\]
\item 
\textbf{Claim $2$:}
~\[p_{m+1}p_{m+2}\cdots p_{m+n}\sigma_{m+n}\cdots\sigma_{m+2}\sigma_{m+1\vert\mathcal{Z}_m}=n!\cdot{\rm id}_{\vert\mathcal{Z}_m}.\]
\end{itemize}
Then we can inductively construct a sequence of commutative diagrams $(\mathcal{D}^{(i)})$ for $0<i< p$
\[
\xymatrix{
\mathfrak{f}^{(i-1)}\ar@{^{(}->}[rr] \ar@{=}[d] &&
\mathfrak{f}^{(i)} \ar@{->>}[rr] \ar@{_{(}->}[d]&&
\mathcal{Z}_{m+p-i}\ar@{_{(}->}[d]\ar@/_0.7pc/[ll]_{\sigma_{m+p,m+p-i}}  \\
\mathfrak{f}^{(i-1)}\ar@{^{(}->}[rr] \ar[d]&&
\mathfrak{h}_{m+p}\ar@{->>}[rr]^{p_{m+p-i,m+p}}\ar@{->>}[d]&&
\mathfrak{h}_{m+p-i}\ar@{->>}[d]\\
0\ar[rr]&&
\mathfrak{h}_{m+p}/\mathfrak{f}^{(i)}\ar[rr]^{\cong}&&
\mathfrak{h}_{m+p-i-1},
}
\]
where all the rows and columns are short exact sequences, 
\[\mathfrak{f}^{(i-1)}\cong \mathcal{Z}_{m+p}\oplus \mathcal{Z}_{m+p-1}\oplus \cdots\oplus \mathcal{Z}_{m+p-i+1},\] and 
\[p_{m+p-i,m+p}=p_{m+p-i+1}\cdots p_{m+p-1}p_{m+p},\] 
\[\sigma_{m+p,m+p-i}=\sigma_{m+p}\sigma_{m+p-1}\cdots\sigma_{m+p-i+1}.\]
Since $p_{m+p-i,m+p}\cdot\sigma_{m+p,m+p-i}=i!\cdot{\rm id}_{\vert\mathcal{Z}_{m+p-i}}$, the top row of $(\mathcal{D}^{(i)})$ splits which induces the diagram $(\mathcal{D}^{(i+1)})$ for $i+1< p$. Then middle column of $(\mathcal{D}^{(p-1)})$ is the desired short exact sequence.
\end{proof}
This proposition has direct implication for the homotopy exponent problem. Let $Y$ be a connected co-$H$-space. Suppose we have a morphism of bi-$\Delta$-groups 
\[\mathcal{G}\rightarrow \Omega \Sigma Y^\ast(Y)=\{[Y^{\times (n+1)}, \Omega \Sigma Y]\}_{n\geq 0}\]
such that there exists some $g\in \mathfrak{h} \mathcal{G}$ as the representative of $[{\rm id}]\in [\Omega \Sigma Y, \Omega \Sigma Y]$.
Further, suppose we have a commutative diagram of groups homomorphisms 
\[
\xymatrix{
\mathfrak{h} \mathcal{G}\ar@{->>}[d] \ar[rr]
&&
[J(Y), J(Y)]
\ar[d]^{\Omega}
\\
\mathfrak{h} (\mathcal{G}/\mathcal{H})
\ar@{.>}[rr]
&&
\lbrack \Omega^2\Sigma Y, \Omega^2\Sigma Y\rbrack,
}
\]
where $\mathcal{H}$ is a normal bi-$\Delta$-subgroup of $\mathcal{G}$ such that $\mathcal{K}=\mathcal{G}/\mathcal{H}$ is an abelian bi-$\Delta$-group.
Then the problem of the $p$-exponent of $[\Omega^2 \Sigma Y, \Omega^2 \Sigma Y]$ is equivalent to a sequence of extension problems for $m\geq 1$:
\[0\rightarrow \mathcal{Z}_m\mathcal{K}\rightarrow \mathfrak{h}_m\mathcal{K}\rightarrow \mathfrak{h}_{m-1}\mathcal{K}\rightarrow 0.\]
We may call these extensions the obstructions to the exponent problem for $[\Omega^2 \Sigma Y, \Omega^2 \Sigma Y]$, and then for $\pi_\ast(\Sigma Y)$.
\begin{corollary}\label{0obstruction}
For the $p$-exponent of the homotopy groups $\pi_\ast(\Sigma Y)$, the first non-trivial obstruction can only appear when $m=p$. Further, it suffices to consider the obstructions when $m=(k+1)p-1$ for $k\geq 1$ associated to extensions
\[0\rightarrow \bigoplus_{i=0}^{p-1}\mathcal{Z}_{kp+i}\mathcal{K}\rightarrow \mathfrak{h}_{(k+1)p-1}\mathcal{K}\rightarrow \mathfrak{h}_{kp-1}\mathcal{K}\rightarrow 0.\]
\end{corollary} 
Now we turn to study some natural idempotents of abelian bi-$\Delta$-groups. Our way of constructing idempotents may be helpful for producing a functorial homotopy decomposition of $\Omega\Sigma Y$. 
\begin{proposition}
Given any abelian bi-$\Delta$-group $\mathcal{G}$ such that the sequence 
\[0\rightarrow \mathcal{Z}_m\mathcal{G}\stackrel{i_m}{\rightarrow} \mathfrak{h}_m\mathcal{G}\rightarrow \mathfrak{h}_{m-1}\mathcal{G}\rightarrow 0\]
splits for some $m$, then the composition
\[e_m^{(n)}: \mathfrak{h}_n\mathcal{G}\stackrel{p_{m,n}}{\longrightarrow} \mathfrak{h}_{m}\mathcal{G} \stackrel{\pi_m}{\longrightarrow} \mathcal{Z}_{m}\mathcal{G} \stackrel{i_m}{\hookrightarrow} \mathfrak{h}_{m}\stackrel{H_{m,n}}{\longrightarrow} \mathfrak{h}_{n\mathcal{G}}\]
is an idempotent for any $n\geq m$, i.e., $e_m^{(n)}\circ e_m^{(n)}=e_m^{(n)}$, where $\pi_m$ is a retraction such that $\pi_m\circ i_m={\rm id}$.
\end{proposition}
\begin{proof} The lemma follows from $p_{m,n}\circ H_{m,n}\circ i_m=i_m$.
\end{proof}
\begin{corollary}
Given any abelian bi-$\Delta$-group $\mathcal{G}$ such that $\mathfrak{h}_i\mathcal{G}=0$ for any $i\leq k$, and $\mathcal{G}$ is $p$-local for some prime $p$, then $e_m^{(n)}$ is an idempotent for any $n\geq m$ and $m\leq p+k$.
\end{corollary}
\begin{corollary}
Given any abelian bi-$\Delta$-group $\mathcal{G}$ such that $\mathfrak{h}\mathcal{G}$ is completely splittable, which means the sequence 
\[0\rightarrow \mathcal{Z}_m\mathcal{G}\stackrel{i_m}{\rightarrow} \mathfrak{h}_m\mathcal{G}\rightarrow \mathfrak{h}_{m-1}\mathcal{G}\rightarrow 0\]
splits for each $m$, then the composition 
\[e_{m,n}: \mathfrak{h}\mathcal{G}  \stackrel{p_{n, \infty}}{\longrightarrow}\mathfrak{h}_n\mathcal{G}\stackrel{e_m^{(n)}}{\longrightarrow} \mathfrak{h}_n\mathcal{G}  \stackrel{i_{\infty, n}}{\longrightarrow}\mathfrak{h}\mathcal{G}\]
is an idempotent for any $n\geq m$, where $p_{n, \infty}i_{\infty, n}={\rm id}$.
\end{corollary}
To close the appendix, we construct some partially null homotopic self maps of double loop suspensions in terms of topological James-Hopf operations which are closely related to the combinatorial versions.
\begin{proposition}\label{lastnull}
For any map $f\in {\rm Sh}_{m}^{N}(\Sigma X; \Omega Z)+{\rm T}_{m}^{N}(\Sigma X; \Omega Z)\subseteq\mathcal{Z}_m\Omega Z^{\ast}(\Sigma X)$, the loop of the composition
\[\Omega\Sigma^2X \stackrel{H_{m+1}}{\longrightarrow} \Omega\Sigma\big((\Sigma X)^{\wedge(m+1)}\big) \stackrel{J(f)}{\longrightarrow}\Omega Z\]
is null homotopic on the natural $(m+1)$-th filtration $\mathcal{F}_{m+1}(X)$ of $\Omega^2\Sigma^2X\simeq \mathcal{F}(X)$, where $H_{m+1}$ is the James-Hopf invariant.
\end{proposition}
\begin{proof}
The proposition follows from Lemma \ref{newnull0} and \ref{2ndnull} and the fact that (\cite{Wu}) 
\[H_m(f):\Omega\Sigma^2X \rightarrow \Omega Z\]
is represented by $J(f)\circ H_{m+1}$.
\end{proof}

\noindent\textbf{Acknowledgements.} 
We want to thank Stephen Theriault most warmly for proofreading this paper. We are also indebted to the referee for his/her careful reading of our manuscript, which has improved the exposition of this paper.

\end{document}